\documentclass[a4paper,10pt]{amsart}
\usepackage{amsmath,amsthm,amssymb,amsfonts,enumerate,color}
\newcommand{\norm}[1]{\left\Vert#1\right\Vert}
\newcommand{\abs}[1]{\left\vert#1\right\vert}
\newcommand{\set}[1]{\left\{#1\right\}}
\newcommand{\Real}{\mathbb{R}}

\newcommand{\supp}{\operatorname{supp}}

\renewcommand{\L}{\mathcal{L}}
\renewcommand{\P}{\mathcal{P}}

\newcommand{\K}{\mathcal{K}}

\newcommand{\R}{\mathcal{R}}
\newcommand{\W}{\mathcal{W}}
\newcommand{\M}{\mathcal{M}}

\newtheorem{thm}{Theorem}[section]
\newtheorem{prop}[thm]{Proposition}

\newtheorem{lem}[thm]{Lemma}

\newtheorem{defn}[thm]{Definition}
\newtheorem{rem}[thm]{Remark}
\numberwithin{equation}{section}
\allowdisplaybreaks

\author[T. Ma]{Tao Ma}
    \address{School of Mathematics and Statistics \\
    Wuhan University\\
    430072 Wuhan, China}
\email{tma.math@whu.edu.cn}

\author[P. R. Stinga]{Pablo Ra\'ul Stinga}
    \address{Department of Mathematics\\
    The University of Texas at Austin\\
    1 University Station, C1200 Austin, TX 78712-1202, United States of America}
    \email{stinga@math.utexas.edu}

\author[J. L. Torrea]{Jos\'e L. Torrea}
    \address{Departamento de Matem\'aticas\\
    Universidad Aut\'onoma de Madrid\\
    28049 Madrid, Spain\\and ICMAT-CSIC-UAM-UCM-UC3M}
\email{joseluis.torrea@uam.es}

\author[C. Zhang]{Chao Zhang}
\address{School of Mathematics and Statistics \\
          Wuhan University \\
          430072 Wuhan, China}
\address{\textit{Current address:}
          \vskip0.01cm Departamento de Matem\'aticas \\
          Facultad de Ciencias \\
          Universidad Au\-t\'o\-no\-ma de Madrid \\
          28049 Madrid, Spain}
\email{zaoyangzhangchao@163.com}

\thanks{Research partially supported by Ministerio de Ciencia e Innovaci\'{o}n de Espa\~{n}a MTM2008-06621-C02-01. The first and fourth authors were partially supported by National Natural Science Foundation of China No.11071190. Second author was also supported by grant COLABORA 2010/01 from Planes Riojanos de I+D+I}

\keywords{Schr\"odinger operators, regularity estimates, Campanato spaces, $T1$ criterion, $BMO$ spaces}

\subjclass[2010]{35J10, 35B65, 26A33, 42B37, 46E35, 42B25}

\begin{document}

\title[Regularity for Schr\"odinger operators via a $T1$ theorem]{Regularity estimates in H\"older spaces for Schr\"odinger operators via a $T1$ theorem}

\begin{abstract}
We derive H\"older regularity estimates for operators associated with a time independent Schr\"odinger operator of the form $-\Delta+V$. The results are obtained by checking a certain condition on the function $T1$. Our general method applies to get regularity estimates for maximal operators and square functions of the heat and Poisson semigroups, for Laplace transform type multipliers and also for Riesz transforms and
negative powers $(-\Delta+V)^{-\gamma/2}$, all of them in a unified way.
\end{abstract}

\maketitle

\section{Introduction and statement of the results}

Regularity estimates for second order differential operators are central in the theory of PDEs. In this context, Sobolev and Schauder estimates are fundamental results. The latter can be seen as boundedness between H\"older spaces of negative powers of operators.

In this paper we study regularity estimates in the H\"older classes $C^{0,\alpha}_\L$, $0<\alpha<1$, of operators associated with the time independent
Schr\"{o}dinger operator in $\Real^n$, $n\geq3$,
$$\L:=-\Delta+V.$$
The nonnegative potential $V$ satisfies a reverse H\"older
inequality for some $q\geq n/2$, see Section \ref{Section:BMO}.

It is well-known that the classical H\"older space
$C^{\alpha}(\Real^n)$ can be identified with the Campanato space $BMO^\alpha$,
see  \cite{campanato} . In the Schr\"odinger case the analogous result was proved by B. Bongioanni, E. Harboure and O.
Salinas in \cite{Bongioanni-Harboure-Salinas-Weighted}. They
identified the H\"older space associated to $\L$ with a Campanato type $BMO^\alpha_\L$ space, see Proposition \ref{Prop:BMO y Calpha} below.
Therefore, in order to study regularity estimates we can take advantage of this characterization. In fact we shall present our results as
boundedness of operators between $BMO_\L^\alpha$ spaces.

The main point of this paper is to give a  simple $T1$ criterion
for boundedness in $BMO^\alpha_\L$ of the so-called
$\gamma$-Schr\"odinger-Calder\'on-Zygmund operators $T$, see Definition
\ref{Defn:Operators}. The advantage of this criterion is that
everything reduces to check a certain condition on the function $T1$. The method is applied to the maximal operators
associated with the semigroups $e^{-t\L}$ and $e^{-t\L^{1/2}}$ ({or more general Poisson operators associated to the extension problem for $\L^\sigma$}), the $\L$-square functions, the Laplace transform type multipliers $m(\L)$, the $\L$-Riesz transforms and the negative powers $\L^{-\gamma/2}$, $\gamma>0$.

We use the notation $\displaystyle f_B=\frac{1}{\abs{B}}\int_Bf$. The first result reads as follows.

\begin{thm}[$T1$ criterion for $BMO^\alpha_\L$, $0<\alpha<1$]\label{Thm:cri BMO alpha}
Let $T$ be a $\gamma$-Schr\"odinger-Calder\'on-Zyg\-mund operator, $\gamma\geq0$, with smoothness exponent $\delta$, such that $\alpha+\gamma<\min\set{1,\delta}$. Then $T$ is bounded from $BMO_\L^\alpha$ into $BMO^{\alpha+\gamma}_\L$ if and only if there exists a constant $C$ such that
$$\left(\frac{\rho(x)}{s}\right)^\alpha\frac{1}{|B|^{1+\frac{\gamma}{n}}}\int_B|T1(y)-(T1)_B|~dy\leq C,$$
for every ball $B=B(x,s)$, $x\in\Real^n$ and $0<s\leq\tfrac{1}{2}\rho(x)$. Here $\rho(x)$ is defined in \eqref{rho}.
\end{thm}

We can also consider the endpoint case $\alpha=0$.

\begin{thm}[$T1$ criterion for $BMO_\L$]\label{Thm:cri BMO}
Let $T$ be a $\gamma$-Schr\"odinger-Calder\'on-Zygmund operator,
$0\le \gamma<\min\set{1,\delta}$, with smoothness exponent $\delta$. Then $T$ is a bounded operator from $BMO_\L$ into
$BMO_\L^\gamma$ if and only if there exists a constant $C$ such that
$$\displaystyle\log\left(\frac{\rho(x)}{s}\right)\frac{1}{|B|^{1+\frac{\gamma}{n}}}\int_B|T1(y)-(T1)_B|~dy\leq C,$$
for every ball $B=B(x,s)$, $x\in\Real^n$ and $0<s\leq\tfrac{1}{2}\rho(x)$.
\end{thm}

Observe that for any $x\in\Real^n$ and $0<\alpha\le1$, if
$0<s\le\tfrac{1}{2}\rho(x)$ then $1+\log\frac{\rho(x)}{s}\sim
\log\frac{\rho(x)}{s}$ and
$1+\frac{2^\alpha\left(\left(\frac{\rho(x)}{s}\right)^\alpha-1\right)}{2^\alpha-1}\sim\left(\frac{\rho(x)}{s}\right)^\alpha$. Therefore, by tracking down the exact constants in the proof we can see that Theorem \ref{Thm:cri BMO} is indeed the limit case of  Theorem \ref{Thm:cri BMO alpha}.

Theorem \ref{Thm:cri BMO} is a generalization of the $T1$-type
criterion given in \cite{BCFST} for the case of the harmonic
oscillator $H=-\Delta+\abs{x}^2$.  Here we require the dimension to
be $n\geq3$, while in \cite{BCFST} the dimension can be any
$n\geq1$.

As a by-product of our main results we are able to characterize
pointwise multipliers of the spaces $BMO^\alpha_\L$, see Proposition
\ref{Prop:Multipliers} below. For pointwise multipliers of the
classical $BMO^\alpha$ spaces see the papers by S. Bloom \cite{Bloom}, S. Janson \cite{Janson} and
E. Nakai and K. Yabuta  \cite{Nakai-Yabuta}.

Next we present the announced applications. For the definitions of
the operators see subsections
\ref{Subsection:Maximal heat} to \ref{Subsection:Laplace}.

\begin{thm}\label{Thm:0-operators}
Let $0\le\alpha<\min\{1,2-\frac{n}{q}\}$. The maximal operators associated with the heat semigroup $\{\W_t\}_{t>0}$ and with the generalized Poisson operators $\{\P_t^\sigma\}_{t>0}$, the Littlewood-Paley $g$-functions given in terms of the heat and the Poisson semigroups, and the Laplace transform type multipliers $m(\L)$, are bounded from $BMO_\L^\alpha$ into itself.
\end{thm}

In \cite{DGMTZ} it was proved that the maximal operator of the heat semigroup, the  maximal operator of the  Poisson semigroup and the square function of the heat semigroup are bounded in $BMO_\L$, and that the fractional integral $\L^{-\gamma/2}$ maps $L^{n/\gamma}(\Real^n)$ into $BMO_\L$, $0<\gamma<n$. The square function was also studied in \cite{Abu-Falahah-Stinga-Torrea}. In \cite{StinTo} it was proved that the fractional integral in the case of the harmonic oscillator has similar boundedness properties in the scale of spaces $BMO^\alpha_H$, or more generally, $C^{k,\alpha}_H(\Real^n)$.

The Riesz transforms associated to $\L$ were introduced and studied
in $L^p(\Real^n)$ in the seminal paper by Z. Shen \cite{Shen}. B.
Bongioanni, E. Harboure and O. Salinas developed their mapping
properties on $BMO^\alpha_\L$ in
\cite{Bongioanni-Harboure-Salinas-Riesz}. They also studied the
corresponding boundedness results for the negative powers, see \cite{Bongioanni-Harboure-Salinas-Weighted}, and $L^p$-boundedness for the commutators with a function, see  \cite{Bongioanni-Harboure-Salinas-2}. Following the pattern
of the proof of Theorem \ref{Thm:0-operators} we can recover the
results from  \cite{Bongioanni-Harboure-Salinas-Weighted} and
\cite{Bongioanni-Harboure-Salinas-Riesz}. We state them as a theorem
for further reference.

\begin{thm}\label{Thm:Pola}
Let $\alpha\geq0$ and $0<\gamma<n$. Then:
\begin{enumerate}[$\bullet$]
    \item The $\L$-Riesz transforms are bounded from $BMO^\alpha_\L$ into itself, for any $0\le\alpha<1-\frac{n}{q}$, with
    $q>n$.
    \item The negative powers $\L^{-\gamma/2}$ are bounded from $BMO_\L^\alpha$ into $BMO^{\alpha+\gamma}_\L$ for $\alpha+\gamma<\min\{1,2-\frac{n}{q}\}$.
\end{enumerate}
\end{thm}

Regarding Sobolev estimates, more general operators can be considered by replacing $-\Delta$ by some second order elliptic operator $A$ with bounded measurable coefficients. When $A$ is a degenerate divergence form elliptic operator, some estimates for the Green function and the heat semigroup were obtained by J. Dziuba\'nski in \cite{Dziubanski}. A priori $L^p$ estimates and global existence and uniqueness results in $L^p$ for the case when $A$ is in nondivergence form with $VMO$ coefficients were found by M. Bramanti, L. Brandolini, E. Harboure and B. Viviani in \cite{Bramanti-Brandolini-Harboure-Viviani}.

The use of the action of an operator $T$ on the function $1$ in order to get some boundedness properties of $T$ goes back to the celebrated work by G. David and J.-L. Journ\'e, see \cite{DJ}. For  vector-valued versions of these criteria, see the papers by T. Hyt\"onen \cite{H} and T. Hyt\"onen and L. Weis  \cite{H2}.

The paper is organized as follows. In section \ref{Section:BMO} we collect the technical results about the space $BMO_\L^\alpha$. Section \ref{Section:Proofs} is devoted to the proofs of the main theorems. The applications are given in section \ref{Section:Applications}. Through the paper the letters $C$ and $c$ denote positive constants that may change at each occurrence and $\mathcal{S}$ is the class of rapidly decreasing $C^\infty$ functions in $\Real^n$.

\section{The spaces $BMO^\alpha_\L$, $0\leq\alpha\leq1$}\label{Section:BMO}

The nonnegative potential $V$ satisfies a reverse H\"older
inequality for some $q\ge \frac{n}{2}$; that is, there exists a
constant $C=C(q,V)$ such that
$$\left(\frac{1}{\abs{B}}\int_BV(y)^q~dy\right)^{1/q}\leq\frac{C}{\abs{B}}\int_BV(y)~dy,$$
for all balls $B\subset\Real^n$. We write $V\in RH_q$. Associated to this potential, Z. Shen defines  the critical radii function in \cite{Shen} as
\begin{equation}\label{rho}
\rho(x):=\sup\Big\{r>0:\frac{1}{r^{n-2}}\int_{B(x,r)}V(y)~dy\leq1\Big\},\qquad x\in\Real^n.
\end{equation}
We have $0<\rho(x)<\infty$.

Let us begin with some properties of the critical radii function $\rho$.

\begin{lem}[See {\cite[Lemma~1.4]{Shen}}]\label{Lem:equiv rho}
There exist $c>0$ and $k_0\geq1$ such that for all $x,y\in\Real^n$,
$$c^{-1}\rho(x)\left(1+\frac{\abs{x-y}}{\rho(x)}\right)^{-k_0}\leq\rho(y)\leq c\rho(x)\left(1+\frac{\abs{x-y}}{\rho(x)}\right)^{\frac{k_0}{k_0+1}}.$$
In particular, there exists a positive constant $C_1<1$ such that if $\abs{x-y}\leq\rho(x)$ then $C_1\rho(x)<\rho(y)<C_1^{-1}\rho(x)$.
\end{lem}

\noindent\textbf{Covering by critical balls.} According to \cite[Lemma~2.3]{Dziubanski-Zienkiewicz} there exists a sequence of points $\set{x_k}_{k=1}^\infty$ in $\Real^n$ such that if $Q_k:=B(x_k,\rho(x_k))$, $k\in\mathbb N$, then
\begin{enumerate}[(a)]
    \item $\cup_{k=1}^\infty Q_k= \mathbb R^n$, and
    \item there exists $N\in \mathbb N$ such that $\operatornamewithlimits{card}\{j\in \mathbb N: Q_j^{**} \cap Q_k^{**} \neq \emptyset \} \leq N$, for every $k \in \mathbb N$.
\end{enumerate}
For a ball $B$, the notation $B^\ast$ above means the ball with the same center as $B$ and twice radius.

The definition of space $BMO_\L$ was given in \cite{DGMTZ}. The space $BMO^\alpha_\L$, $0<\alpha\leq1$, was introduced in \cite{Bongioanni-Harboure-Salinas-Weighted}. We collect from there the following facts.

A locally integrable function $f$ in $\Real^n$ is in $BMO^\alpha_\L$, $0\le \alpha\le 1$ provided there exists $C>0$ such that
\begin{enumerate}
    \item[(i)] $\displaystyle\frac{1}{|B|}\int_B|f(x)-f_B|~dx\leq C\abs{B}^\frac{\alpha}{n}$, for every ball $B$ in $\Real^n$, and
    \item[(ii)] $\displaystyle\frac{1}{|B|}\int_B|f(x)|~dx\leq C\abs{B}^\frac{\alpha}{n}$, for every
 $B=B(x_0,r_0)$, where $x_0 \in \mathbb R^n$ and $ r_0 \geq \rho(x_0)$.
\end{enumerate}
The norm $\|f\|_{BMO^\alpha_\L}$ of $f$ is defined as the minimum $C>0$ such that (i) and (ii) above hold. We have $BMO_\L^0=BMO_\L$.

By using the classical John-Nirenberg inequality it can be seen that if in (i) and (ii) $L^1$-norms are replaced by $L^p$-norms, for $1<p<\infty$, then the space $BMO^\alpha_\L$ does not change and equivalent norms appear. In this case the conditions read:
\begin{enumerate}
\item[(i)$_p$] $\displaystyle\left(\frac{1}{\abs{B}}\int_B\abs{f(x)-f_B}^p~dx\right)^{1/p}
\leq C\abs{B}^{\frac{\alpha}{n}}$, for every ball $B$ in $\Real^n$, and
\item[(ii)$_p$] $\displaystyle\left(\frac{1}{\abs{B}}\int_B\abs{f(x)}^p~dx\right)^{1/p}
\leq C\abs{B}^{\frac{\alpha}{n}}$, for every $B=B(x_0,r_0)$, where $x_0\in\Real^n$ and $r_0\geq\rho(x_0)$.
\end{enumerate}

Let us note that if (ii) (resp. (ii)$_p$) above is true for some ball $B$ then (i) (resp. (i)$_p$) holds for the same ball, so we might ask to (i) (resp. (i)$_p$) only for balls with radii smaller than $\rho(x)$.

The restriction $\alpha\leq1$ in the definition above is necessary because if $\alpha>1$ then the space $BMO^\alpha_\L$ only contains constant functions.

\begin{prop}\label{Prop:Properties BMO}
Let $B=B(x,r)$ with $r<\rho(x)$.
\begin{enumerate}[(1)]
    \item (See {\cite[Lemma~2]{DGMTZ}}) If $f\in BMO_\L$ then $\abs{f_B}\leq C\left(1+\log\frac{\rho(x)}{r}\right)\norm{f}_{BMO_\L}$.
    \item (See {\cite[Proposition~4.3]{Ma-Stinga-Torrea-Zhang}}) If $f\in BMO_\L^\alpha$, $0<\alpha\leq1$, then we have $\abs{f_B}\leq C_{\alpha}\norm{f}_{BMO_\L^\alpha}\rho(x)^\alpha$.
    \item (See {\cite[Proposition~3]{Bongioanni-Harboure-Salinas-Weighted}}) A function $f$ belongs to $BMO^\alpha_\L$, $0\leq\alpha\leq1$, if and only if $f$ satisfies $\mathrm{(i)}$ for every ball $B=B(x_0,r_0)$ with $r_0<\rho(x_0)$ and $\abs{f}_{Q_k}\leq C\abs{Q_k}^{1+\frac{\alpha}{n}}$, for all balls $Q_k$ given in the covering by critical balls above.
\end{enumerate}
\end{prop}

\begin{lem}[Boundedness criterion]\label{Lem:bound cri}
Let $S$ be a linear operator defined on $BMO^\alpha_\L$, $0\le \alpha\le 1$. Then $S$ is bounded from $BMO^\alpha_\L$ into $BMO^{\alpha+\gamma}_\L$, $\alpha+\gamma\leq1$, $\gamma\geq0$, if there exists $C>0$ such that, for every $f\in BMO^\alpha_\L$ and $k\in\mathbb N$,
\begin{enumerate}
    \item[($A_k$)] $\displaystyle \frac{1}{|Q_k|^{1+\frac{\alpha+\gamma}{n}}}\int_{Q_k}|Sf(x)|~dx \leq C\|f\|_{BMO^\alpha_\L}$, and
    \item[($B_k$)] $\displaystyle \|Sf\|_{BMO^{\alpha+\gamma}(Q_k^*)}\leq C\|f\|_{BMO^\alpha_\L}$, where $BMO^\alpha(Q_k^*)$ denotes the usual $BMO^\alpha$ space on the ball $Q_k^*$.
\end{enumerate}
\end{lem}

\begin{proof}
For $\alpha=0$ the result is already contained in \cite[p.~346]{DGMTZ}. The general statement follows immediately from the definition of $BMO^\alpha_\L$ and Lemma \ref{Lem:equiv rho} (see Proposition \ref{Prop:Properties BMO}).
\end{proof}

The duality of the $\L$-Hardy space $H^1_\L$ with $BMO_\L$ was proved in \cite{DGMTZ}. As already mentioned in the paper by Bongioanni, Harboure and Salinas \cite{Bongioanni-Harboure-Salinas-Weighted}, the $BMO^\alpha_\L$ spaces are the duals of the $H^p_\L$ spaces defined in \cite{Dziubanski-Zienkiewicz,Dziubanski-Zienkiewicz Colloquium,Dziubanski-Zienkiewicz Hp}. In fact, if $s>n$ and $0\leq\alpha<1$ then the dual of $H^{\frac{n}{n+\alpha}}_\L$ is $BMO^\alpha_\L$; see also \cite{H-S-V}, references in \cite{Ma-Stinga-Torrea-Zhang} and \cite{Yang}.

We denote by $C^\alpha(\Real^n)$ the space of $\alpha$-H\"older continuous functions on $\Real^n$ and by $[f]_{C^\alpha}$ its usual seminorm. Recall that $BMO^\alpha(\Real^n)=C^\alpha(\Real^n)$ with $\norm{f}_{BMO^\alpha(\Real^n)}\sim[f]_{C^\alpha}$.

\begin{prop}[Campanato description, {\cite[Proposition~4]{Bongioanni-Harboure-Salinas-Weighted}}]\label{Prop:BMO y Calpha}
Let $0<\alpha\leq1$. A function $f$ belongs to $BMO_\L^\alpha$ if and only if $f\in C^\alpha(\Real^n)$ and $\abs{f(x)}\leq C\rho(x)^\alpha$, for all $x\in\Real^n$. Moreover, $\norm{f}_{BMO^\alpha_\L}\sim[f]_{C^\alpha(\Real^n)}+\norm{f\rho^{-\alpha}}_{L^\infty(\Real^n)}$.
\end{prop}

In the following lemma we present examples of families of functions indexed by $x_0\in\Real^n$ and $0<s\leq\rho(x_0)$ that are uniformly bounded in $BMO^\alpha_\L$. They will be useful in the sequel.

\begin{lem}\label{Lem:Examples}
There exists constants $C,C_\alpha>0$ such that for every $x_0\in\Real^n$ and $0<s\leq\rho(x_0)$,
\begin{enumerate}[(a)]
    \item the function
        $$g_{x_0,s}(x):=\chi_{[0,s]}(|x-x_0|)\log\left(\frac{\rho(x_0)}{s}\right)+ \chi_{(s,\rho(x_0)]}(|x-x_0|)\log\left(\frac{\rho(x_0)}{|x-x_0|}\right),$$
        $x\in\Real^n$, belongs to $BMO_\L$ and $\norm{g_{x_0,s}}_{BMO_\L}\leq C$;
    \item the function
        \begin{align*}
        f_{x_0,s}(x)&=\chi_{[0,s]}(|x-x_0|)\left(\rho(x_0)^\alpha-s^\alpha\right) \\
         &\quad+ \chi_{(s,\rho(x_0)]}(|x-x_0|)\left({\rho(x_0)}^\alpha-{|x-x_0|}^\alpha\right),
        \end{align*}
        $x\in\Real^n$, belongs to $BMO^\alpha_\L$, $0<\alpha\le 1$, and $\norm{f_{x_0,s}}_{BMO^\alpha_\L}\leq C_\alpha$.
\end{enumerate}
\end{lem}

\begin{proof}
The proof of \textit{(a)} follows the same lines as the proof of Lemma 2.1 in \cite{BCFST}. We omit the details.

Let us continue with \textit{(b)}. Recall that the function $h(x)=\left(1-\abs{x}^\alpha\right)\chi_{[0,1]}(\abs{x})$ is in $BMO^\alpha(\mathbb R^n)$. Hence, for every $R>0$, the function $h_R(x):=R^\alpha h(x/R)$ is in $BMO^\alpha(\mathbb R^n)$ and $\norm{h_R}_{BMO^\alpha(\Real^n)}\leq C$, where $C>0$ is independent of $R$. Moreover, for every $R>0$ and $S\ge1$, the function $h_{R,S}(x)=\min\{R^\alpha(1-S^{-\alpha}),R^\alpha h(x/R)\}$ belongs to $BMO^\alpha(\mathbb R^n)$ and $\|h_{R,S}\|_{BMO^\alpha(\mathbb R^n)}\leq C$, where $C>0$ does not depend on $R$ and $S$. Then, since for every $x_0\in\Real^n$ and $0<s\leq\rho(x_0)$,
$$f_{x_0,s}(x)=h_{\rho(x_0),\frac{\rho(x_0)}{s}}(x-x_0),\quad x\in\Real^n,$$
we get $f_{x_0,s}\in BMO^\alpha(\mathbb R^n)=C^\alpha(\Real^n)$
and $\norm{f_{x_0,s}}_{BMO^\alpha(\mathbb R^n)}\leq C$. This, the
obvious inequality $\abs{f_{x_0,s}(x)}\leq C\rho(x)^\alpha$, for all
$x$, uniformly in $x_0$ and $s\leq\rho(x_0)$, and Proposition \ref{Prop:BMO y Calpha} imply the conclusion.
\end{proof}

\section{Operators and proofs of the main results}\label{Section:Proofs}

\subsection{The operators related to $\L$}

We denote by $L^p_c(\mathbb R^n)$ the set of functions $f\in L^p(\Real^n)$, $1\leq p\leq\infty$, whose support $\supp(f)$ is a compact subset of $\Real^n$.

\begin{defn}\label{Defn:Operators}\normalfont
Let $0\leq\gamma<n$, $1<p\leq q<\infty$, $\tfrac{1}{q}=\tfrac{1}{p}-\tfrac{\gamma}{n}$. Let $T$ be a bounded linear operator from $L^p(\Real^n)$ into $L^q(\Real^n)$ such that
$$Tf(x)=\int_{\Real^n}K(x,y)f(y)~dy,\quad f\in L^p_c(\Real^n)~\hbox{and a.e.}~x\notin\supp(f).$$
We shall say that $T$ is a $\gamma$-Schr\"odinger-Calder\'on-Zygmund operator {with regularity exponent $\delta>0$} if for some constant $C$
\begin{enumerate}[(1)]
  \item $\displaystyle\abs{K(x,y)}\leq\frac{C}{|x-y|^{n-\gamma}}\left(1+\frac{\abs{x-y}}{\rho(x)}\right)^{-N}$, for all $N>0$ and $x\neq y$,
  \item $\displaystyle|K(x,y)-K(x,z)|+|K(y,x)-K(z,x)|\leq C\frac{|y-z|^\delta}{|x-y|^{n-\gamma+\delta}}$, when $|x-y|>2|y-z|$.
\end{enumerate}
\end{defn}

\noindent\textbf{Definition of $Tf$ for $f\in BMO_\L^\alpha$, $0\le \alpha\le 1$.} Suppose that $f \in BMO^\alpha_\L$ and $R\geq\rho(x_0)$, $x_0\in\Real^n$. We define
$$Tf(x)=T\left(f\chi_{B(x_0,R)}\right)(x)+\int_{B(x_0,R)^c} K(x,y)f(y)~dy,\quad\hbox{a.e.}~x\in B(x_0,R).$$
Note that the first term in the right hand side makes sense since $f\chi_{B(x_0,R)}\in L^p_c(\Real^n)$. The integral in the second term is absolutely convergent. Indeed, by Lemma \ref{Lem:equiv rho}, there exists a constant $C$ such that for any $x\in B(x_0,R)$,
\begin{align*}
    \rho(x) &\le c\rho(x_0)\left(1+\frac{\abs{x-x_0}}{\rho(x_0)}\right)^{\frac{k_0}{k_0+1}}\le C\left(\rho(x_0)+\rho(x_0)^{1-\frac{k_0}{k_0+1}}\abs{x-x_0}^{\frac{k_0}{k_0+1}}\right) \\
     &\le C\left(R+R^{1-\frac{k_0}{k_0+1}}\abs{x-x_0}^{\frac{k_0}{k_0+1}}\right)\leq C2R.
\end{align*}
Hence, using the $\gamma$-Schr\"odinger-Calder\'on-Zygmund condition (1) for $K$ with with $N-\gamma>\alpha$,
\begin{align}\label{est K}
    \int_{B(x_0,2R)^c}|K(x,y)||f(y)|~dy &\leq C \sum^\infty_{j=1} \int_{2^jR<|y-x_0|\leq 2^{j+1}R}\frac{\rho(x)^N}{|x-y|^{n+N-\gamma}}~|f(y)|~dy\nonumber \\
     &\leq C \sum^\infty_{j=1}\frac{\rho(x)^N}{(2^jR-R)^{n+N-\gamma}}\int_{|y-x_0|\leq 2^{j+1}R}|f(y)|~dy \\
     &\leq C{R^{\alpha+\gamma}} \|f\|_{BMO^\alpha_\L},\quad\hbox{a.e.}~x\in B(x_0,R).\nonumber
\end{align}
The definition of $Tf(x)$ is also independent of $R$ in the sense that if $B(x_0,R)\subset B(x_0',R')$, with $R'\geq\rho(x_0)$, then the definition using $B(x_0',R')$ coincides almost everywhere in $B(x_0,R)$ with the one just given, because in that situation,
\begin{align*}
    T&\left(f\chi_{B(x_0',R')}\right)(x)-T\left(f\chi_{B(x_0,R)}\right)(x) \\
     &= T\left(f\chi_{B(x_0',R')\setminus B(x_0,R)}\right)(x)=\int_{B(x_0',R')\setminus B(x_0,R)}K(x,y)f(y)~dy \\
     &= \int_{B(x_0,R)^c} K(x,y)f(y)~dy-\int_{B(x_0',R')^c}K(x,y)f(y)~dy,.
\end{align*}
for almost every $x\in B(x_0,R)$.

The definition just given above is equally valid for $f\equiv1\in BMO_\L$.

Next we derive an expression for $Tf$ where $T1$ appears that will be useful in the proof of our main results. Let $x_0\in\Real^n$ and $r_0>0$. For $B=B(x_0,r_0)$ we clearly have
\begin{equation}\label{descomposicion f}
f=(f-f_B)\chi_{B^{***}}+(f-f_B)\chi_{(B^{***})^c}+f_B=:f_1+f_2+f_3.
\end{equation}
Let us choose $R\ge\rho(x_0)$ such that $B^{***}\subset B(x_0,R)$. By using the definition of $Tf$ given above, the identity in \eqref{descomposicion f}, adding and subtracting $f_B$ in the integral over $B(x_0,R)^c$ and collecting terms we get
\begin{align}\label{defBMO}
    Tf(x) &= T\left(f\chi_{B(x_0,R)}\right)(x)+\int_{B(x_0,R)^c}K(x,y)f(y)~dy \nonumber \\
     &= T\left((f-f_B)\chi_{B^{***}}\right)(x)+T\left((f-f_B)\chi_{B(x_0,R)\setminus B^{***}}\right)(x) \nonumber \\
     &\quad +f_BT\left(\chi_{B(x_0,R)}\right)(x) \nonumber \\
     &\quad +\int_{B(x_0,R)^c}K(x,y)(f(y)-f_B)~dy+f_B\int_{B(x_0,R)^c}K(x,y)~dy \nonumber \\
     &= T\left((f-f_B)\chi_{B^{***}}\right)(x)+\int_{(B^{***})^c}K(x,y)(f(y)-f_B)~dy \\
     &\quad +f_B T1(x),\quad\hbox{a.e.}~x\in B^{***}. \nonumber
\end{align}

We observe that there exists a constant $C$ such that
\begin{equation}\label{mean bound gamma}
\frac{1}{\abs{B}^{1+\frac{\gamma}{n}}}\int_{B}\abs{T1(y)}~dy\leq C,\quad\hbox{for all}~B=B(x,\rho(x)),~x\in\Real^n.
\end{equation}
Indeed, by H\"older's inequality and the $L^p-L^q$ boundedness of $T$,
\begin{align*}
    \frac{1}{\abs{B}^{1+\frac{\gamma}{n}}}\int_B\abs{T\left(\chi_{B^\ast}\right)(y)}~dy &\leq \frac{1}{\abs{B}^{\frac{1}{q}+\frac{\gamma}{n}}}\left(\int_B\abs{T\left(\chi_{B^\ast}\right)(y)}^q~dy\right)^{1/q} \\
     &\leq C\frac{\abs{B}^{1/p}}{\abs{B}^{\frac{1}{q}+\frac{\gamma}{n}}}=C.
\end{align*}
By the integral representation of $T$ and the size condition (1) on $K$ with $N=n+\gamma$, for $y\in B(x,\rho(x))$ we have
\begin{align*}
    \abs{T\left(\chi_{(B^\ast)^c}\right)(y)} &\leq C\sum_{k=1}^\infty\int_{2^j\rho(x)\leq\abs{x-z}<2^{j+1}\rho(x)}\frac{\rho(y)^{n+\gamma}}{\abs{y-z}^{2n}}~dz \\
     &\leq C\rho(y)^{n+\gamma}\sum_{k=1}^\infty\frac{(2^{j+1}\rho(x))^n}{(2^j\rho(x)-\rho(x))^{2n}}\leq C\rho(x)^\gamma,
\end{align*}
because $\rho(x)\sim\rho(y)$. Thus \eqref{mean bound gamma} follows by linearity.

\subsection{Proofs of Theorems \ref{Thm:cri BMO alpha} and \ref{Thm:cri BMO}}

\begin{proof}[Proof of Theorem \ref{Thm:cri BMO alpha}]
First we shall see that the condition on $T1$ implies that $T$ is bounded from $BMO^\alpha_\L$ into $BMO^{\alpha+\gamma}_\L$. In order to do this, we will show that there exists $C>0$ such that the properties $(A_k)$ and $(B_k)$ stated in Lemma \ref{Lem:bound cri} hold for every $k\in\mathbb N$ and $f\in BMO^\alpha_\L$.

We begin with $(A_k)$. According to \eqref{defBMO} with $B=Q_k$,
\begin{align*}
Tf(x)&= T\left((f-f_{Q_k})\chi_{Q_k^{***}}\right)(x)+\int_{(Q_k^{***})^c}K(x,y)(f(y)-f_{Q_k})~dy \\
&\quad +f_{Q_k} T1(x),\quad\hbox{a.e.}~x\in Q_k.
\end{align*}
As $T$ maps $L^p(\Real^n)$ into $L^q(\Real^n)$, $\tfrac{1}{q}=\tfrac{1}{p}-\tfrac{\gamma}{n}$, by H\"older's inequality,
\begin{align*}
    \frac{1}{|Q_k|^{1+\frac{\alpha+\gamma}{n}}}&\int_{Q_k}\left|T\left((f-f_{Q_k})\chi_{Q_k^{***}}\right)(x)\right|dx \\
     &\leq \frac{1}{|Q_k|^{\frac{1}{q}+\frac{\alpha+\gamma}{n}}}\left(\int_{Q_k}\left| T\left((f-f_{Q_k})\chi_{Q_k^{***}}\right)(x)\right|^qdx\right)^{1/q}\\
     &\leq\frac{C}{|Q_k|^\frac{\alpha}{n}}\left(\frac{1}{|Q_k|} \int_{Q_k^{***}}\left|f(x)-f_{Q_k}\right|^pdx\right)^{1/p}\leq C\|f\|_{BMO^\alpha_\L}.
\end{align*}
On the other hand, given $x\in Q_k$, we have $\rho(x)\sim\rho(x_k)$ and if $\abs{x_k-y}>2^j\rho(x_k)$, $j\in\mathbb{N}$, then $\abs{x-y}\ge2^{j-1}\rho(x_k)$. By the size condition (1) of the kernel $K$, for any $N>\alpha$ we have
\begin{align*}
    \frac{1}{\abs{Q_k}^{\frac{\alpha+\gamma}{n}}}&\left|\int_{(Q_k^{***})^c}K(x,y)\big(f(y)-f_{Q_k}\big)~ dy\right| \\
     &\le\frac{1}{\abs{Q_k}^{\frac{\alpha+\gamma}{n}}}\int_{(Q_k^{***})^c}\abs{K(x,y)}\abs{f(y)-f_{Q_k}}~dy \\
     &\le \frac{C}{\abs{Q_k}^{\frac{\alpha+\gamma}{n}}}\int_{(Q_k^{***})^c}\frac{1}{\abs{x-y}^{n-\gamma}}\left(1+\frac{\abs{x-y}}{\rho(x)} \right)^{-N}\abs{f(y)-f_{Q_k}}~dy \\
     &\le \frac{C}{\abs{Q_k}^{\frac{\alpha+\gamma}{n}}}\sum_{j=3}^{\infty}\int_{2^j\rho(x_k)<\abs{x_k-y}\le2^{j+1}\rho(x_k)} \frac{\rho(x)^N}{\abs{x-y}^{n-\gamma+N}}\abs{f(y)-f_{Q_k}}~dy \\
     &\le \frac{C}{\rho(x_k)^\alpha}\sum_{j=3}^{\infty}\frac{\rho(x_k)^N}{\left(2^j\rho(x_k)\right)^{n+N}} \int_{\abs{x_k-y}\le2^{j+1}\rho(x_k)}\abs{f(y)-f_{Q_k}}~dy \\
     &\le C\sum_{j=3}^{\infty}2^{-j(N-\alpha)}(j+1)\norm{f}_{BMO^\alpha_\L}\leq C\|f\|_{BMO_\L^\alpha}.
\end{align*}
Finally, by \eqref{mean bound gamma},
$$\frac{1}{{|Q_k|^{1+\frac{\alpha+\gamma}{n}}}}\int_{Q_k}\abs{f_{Q_k}T1(x)}~dx=\frac{|f_{Q_k}|}{|Q_k|^{\frac{\alpha}{n}}} \frac{1}{|Q_k|^{1+\frac{\gamma}{n}}}\int_{Q_k}|T1(x)|~dx\leq C\|f\|_{BMO^\alpha_\L}.$$
Hence, we conclude that $(A_k)$ holds for $T$ with a constant $C$ that does not depend on $k$.

Let us continue with $(B_k)$. Let $B=B(x_0,r_0)\subseteq Q_k^*$, where $x_0\in\Real^n$ and $r_0>0$. Note that if $r_0\geq\tfrac{1}{2}\rho (x_0)$ then $\rho(x_0)\sim\rho(x_k)\sim r_0$, so proceeding as above we have
$$\frac{1}{|B|^{1+\frac{\alpha+\gamma}{n}}}\int_B\abs{Tf(x)-(Tf)_{B}}~dx\leq\frac{2}{|B|^{1+\frac{\alpha+\gamma}{n}}}\int_B |Tf(x)|~dx \leq C\|f\|_{BMO^\alpha_\L}.$$
Assume next that $0<r_0<\tfrac{1}{2}\rho(x_0)$. We have
\begin{align*}
    \frac{1}{|B|^{1+\frac{\alpha+\gamma}{n}}}&\int_B|Tf(x)-(Tf)_{B}|~dx \\
     &\leq \frac{1}{|B|^{1+\frac{\alpha+\gamma}{n}}}\int_B\frac{1}{|B|}\int_B|Tf_1(x)-Tf_1(z)|~dz~dx \\
     &\quad +\frac{1}{|B|^{1+\frac{\alpha+\gamma}{n}}}\int_B\frac{1}{|B|}\int_B|F_2(x)-F_2(z)|~dz~dx \\
     &\quad +\frac{1}{|B|^{1+\frac{\alpha+\gamma}{n}}}\int_B|Tf_3(x)-(Tf_3)_B|~dx=:L_1+L_2+L_3,
\end{align*}
where $f=f_1+f_2+f_3$ as in \eqref{descomposicion f} and we defined
$$F_2(x)=\int_{(B^{***})^c}K(x,y)f_2(y)~dy,\quad x\in B.$$
Again H\"older's inequality and $L^p-L^q$ boundedness of $T$ give $L_1\leq C\norm{f}_{BMO^\alpha_\L}$. Let us estimate $L_2$. Take $x,z\in B$ and $y\in(B^{***})^c$. Then $8r_0<\abs{y-x_0}\leq\abs{y-x}+r_0$ and therefore $2\abs{x-x_0}<4r_0<\abs{y-x}$. Under these conditions we can apply the smoothness of the kernel (recall Definition \ref{Defn:Operators}(2)) and the restriction $\alpha+\gamma<\min\set{1,\delta}$ to get
\begin{align*}
    &\frac{1}{\abs{B}^{\frac{\alpha+\gamma}{n}}}\abs{F_2(x)-F_2(z)} \leq \frac{C}{r_0^{\alpha+\gamma}}\int_{(B^{***})^c}\abs{K(x,y)-K(z,y)}\abs{f(y)-f_B}~dy \\
     &\leq \frac{C}{r_0^{\alpha+\gamma}}\sum_{j=3}^\infty\int_{2^jr_0\leq\abs{x_0-y}<2^{j+1}r_0} \frac{\abs{x-z}^\delta}{\abs{x-y}^{n-\gamma+\delta}}\abs{f(y)-f_B}~dy \\
     &\leq \frac{C}{r_0^{\alpha+\gamma}}\sum_{j=3}^\infty\frac{r_0^\delta}{((2^j-1)r_0)^{n-\gamma+\delta}}\int_{2^jr_0\leq\abs{x_0-y}<2^{j+1}r_0} \abs{f(y)-f_B}~dy \\
     &\leq C\sum_{j=3}^\infty\frac{2^{-j(\delta-(\alpha+\gamma))}}{(2^{j+1}r_0)^{n+\alpha}}\int_{\abs{x_0-y}<2^{j+1}r_0}\abs{f(y)-f_B}~dy \\
     &= C\sum_{j=3}^\infty\frac{2^{-j(\delta-(\alpha+\gamma))}}{(2^{j+1}r_0)^{n+\alpha}}\int_{\abs{x_0-y}<2^{j+1}r_0}\Big|f(y)-f_{2^{j+1}B} +\sum_{k=0}^j (f_{2^{k+1}B}-f_{2^kB})\Big|\,dy \\
     &\leq C\sum_{j=3}^\infty2^{-j(\delta-(\alpha+\gamma))} \Bigg[\frac{1}{(2^{j+1}r_0)^{n+\alpha}}\int_{\abs{x_0-y}<2^{j+1}r_0}|f(y)-f_{2^{j+1}B}|\,dy \\
     &\qquad +\frac{1}{(2^{j+1}r_0)^{\alpha}}\sum_{k=0}^j\frac{|2^{k+1}B|}{|2^kB|}~\frac{1}{|2^{k+1}B|} \int_{2^{k+1}B}|f(y)-f_{2^{k+1}B}|\,dy\Bigg] \\
     &\leq C\sum_{j=3}^\infty2^{-j(\delta-(\alpha+\gamma))} \Bigg[\|f\|_{BMO^\alpha_\L}+\sum_{k=0}^j~\frac{1}{|2^{k+1}B|^{1+\frac{\alpha}{n}}} \int_{2^{k+1}B}|f(y)-f_{2^{k+1}B}|\,dy\Bigg] \\
     &\leq C\norm{f}_{BMO^\alpha_\L}\sum_{j=3}^\infty2^{-j(\delta-(\alpha+\gamma))}(j+2)=C\norm{f}_{BMO^\alpha_\L}.
\end{align*}
Therefore,  $L_2\leq C\norm{f}_{BMO^\alpha_\L}$. We finally consider $L_3$. Using Proposition \ref{Prop:Properties BMO}\textit{(2)} and the assumption on $T1$ it follows that
\begin{align}\label{L3}
    L_3 &= \frac{\abs{f_B}}{|B|^{1+\frac{\alpha+\gamma}{n}}}\int_B\abs{T1(x)-(T1)_B}~dx \nonumber\\
     &\leq C\norm{f}_{BMO^\alpha_\L}\left(\frac{\rho(x_0)}{r_0}\right)^\alpha\frac{1}{|B|^{1+\frac{\gamma}{n}}}\int_B\abs{T1(x)-(T1)_{B}}~dx \\
     &\leq C\norm{f}_{BMO^\alpha_\L}.\nonumber
\end{align}
This concludes the proof of $(B_k)$. Hence $T$ is bounded from $BMO^\alpha_\L$ into $BMO^{\alpha+\gamma}_\L$.

Let us now prove the converse statement. Suppose that $T$ is bounded from $BMO^\alpha_\L$ into $BMO^{\alpha+\gamma}_\L$. Let $x_0\in\Real^n$ and $0<s\le\tfrac{1}{2}\rho(x_0)$ and $B=B(x_0,s)$. For such $x_0$ and $s$ consider the nonnegative function $f_0(x)\equiv f_{x_0,s}(x)$ defined in Lemma \ref{Lem:Examples}. Using the decomposition $f_0=(f_0-(f_0)_B)\chi_{B^{***}}+(f_0-(f_0)_B)\chi_{(B^{***})^c}+(f_0)_B=:f_1+f_2+(f_0)_B$ we can write $(f_0)_BT1(y)=Tf_0(y)-Tf_1(y)-Tf_2(y)$, so
$$(f_0)_B\frac{1}{\abs{B}^{1+\frac{\alpha+\gamma}{n}}}\int_B\abs{T1(y)-T1_B}~dy\leq\sum_{i=0}^2\frac{1}{\abs{B}^{1+\frac{\alpha+\gamma}{n}}} \int_B\abs{Tf_i(y)-(Tf_i)_B}~dy.$$
We can check that each of the three terms above is controlled by $C\norm{f_0}_{BMO^\alpha_\L}\leq C$, where $C$ is independent of $x_0$ and $s$. Indeed, the case $i=0$ follows by the hypothesis about the boundedness of $T$. For $i=1$ the estimate follows, as usual, by H\"older's inequality and $L^p-L^q$ boundedness of $T$. The term for $i=2$ is done as $L_2$ above. Thus, since $(f_0)_B=C(\rho(x_0)^\alpha-s^\alpha)$ we obtain
$$\left(\frac{\rho(x_0)}{s}\right)^\alpha\frac{1}{\abs{B}^{1+\frac{\gamma}{n}}}\int_B\abs{T1(y)-(T1)_B}~dy\leq C.$$
\end{proof}

\begin{proof}[Proof of Theorem \ref{Thm:cri BMO}]
The proof is the same as the proof of Theorem \ref{Thm:cri BMO alpha} putting $\alpha=0$ everywhere, except for just two differences. The first one is the estimate of the term $L_3$, where we must apply Proposition \ref{Prop:Properties BMO}\textit{(1)} instead of \textit{(2)}. The second difference is the proof of the converse, where instead of $f_{x_0,s}(x)$ we have to consider the function $g_{x_0,s}(x)$ of Lemma \ref{Lem:Examples}.
\end{proof}

\subsection{Pointwise multipliers in $BMO^\alpha_\L$, $0\leq\alpha<1$}

\begin{prop}\label{Prop:Multipliers}
Let $\psi$ be a measurable function on $\Real^n$. We denote by $T_\psi$ the multiplier operator defined by $T_\psi(f)=f\psi$. Then
\begin{enumerate}[(A)]
    \item $T_\psi$ is a bounded operator in $BMO_\L$ if and only if $\psi\in L^\infty(\Real^n)$ and there exists $C>0$ such that, for all balls $B=B(x_0,s)$ with $0<s<\tfrac{1}{2}\rho(x_0)$,
        $$\log\left(\frac{\rho(x_0)}{s}\right)\frac{1}{|B|}\int_B|\psi(y)-\psi_B|~dy\leq C.$$
    \item $T_\psi$ is a bounded operator in $BMO_\L^\alpha$, $0<\alpha<1$, if and only if $\psi\in L^\infty(\Real^n)$ and there exists $C>0$ such that, for all balls $B=B(x_0,s)$ with $0<s<\tfrac{1}{2}\rho(x_0)$,
        $$\left(\frac{\rho(x_0)}{s}\right)^\alpha\frac{1}{|B|}\int_B|\psi(y)-\psi_B|~dy\leq C.$$
\end{enumerate}
\end{prop}

\begin{rem}\normalfont
If $\psi\in C^{0,\beta}(\Real^n)\cap L^\infty(\Real^n)$, $0<\beta\leq1$, then $T_\psi$ is bounded on $BMO_\L$. Moreover, if for some $\gamma$-Schr\"odinger-Calder\'on-Zygmund operator $T$ we have that $T1$ defines a pointwise multiplier in $BMO^\alpha_\L$ then the proposition above and Theorems \ref{Thm:cri BMO} and \ref{Thm:cri BMO alpha} imply that $T$ is a bounded operator on $BMO^\alpha_\L$.
\end{rem}

\begin{proof}[Proof of Proposition \ref{Prop:Multipliers}]
Let us first prove \textit{(B)}. Suppose that $T_\psi$ is a bounded
operator on $BMO^\alpha_\L$, $0<\alpha<1$. For the function
$f_{x_0,s}(x)$ defined in Lemma \ref{Lem:Examples} and any ball
$B=B(x_0,s)$ with $0<s\leq\tfrac{1}{2}\rho(x_0)$, by Proposition
\ref{Prop:Properties BMO}\textit{(2)} applied to $f\psi$ and the
hypothesis, we get
\begin{align*}
    \left(\frac{\rho(x_0)}{s}\right)^\alpha&\frac{1}{|B|}\int_B|\psi(x)|~dx \\
     &\le C_\alpha\frac{\left(\rho(x_0)^\alpha-s^\alpha\right)}{|B|^{1+\frac{\alpha}{n}}}\int_B|\psi(x)|~dx =\frac{C_\alpha}{|B|^{1+\frac{\alpha}{n}}} \int_B|\psi(x)f_{x_0,s}(x)|~dx \\
     &\leq \frac{C_\alpha}{|B|^{1+\frac{\alpha}{n}}}\int_B|(\psi f_{x_0,s})(x)-(\psi f_{x_0,s})_B|~dx+\frac{C_\alpha}{|B|^{\frac{\alpha}{n}}}(\psi f_{x_0,s})_B \\
     &\leq C_\alpha\norm{f_{x_0,s}}_{BMO_\L^\alpha}+C_\alpha \left(\frac{\rho(x_0)}{s}\right)^\alpha\|\psi f_{x_0,s}\|_{BMO_\L^\alpha} \\
     &\leq C_\alpha\left(\frac{\rho(x_0)}{s}\right)^\alpha\|f_{x_0,s}\|_{BMO_\L^\alpha}\leq C\left(\frac{\rho(x_0)}{s}\right)^\alpha.
\end{align*}
Hence $|\psi|_{B}\leq C$ with $C$ independent of $B$, so that $\psi$ is bounded. Next we check the condition on $\psi$. We have
\begin{align*}
    \left(\frac{\rho(x_0)}{s}\right)^\alpha\frac{1}{|B|}&\int_B|\psi(x)-\psi_B|~dx \le C_\alpha\frac{\left(\rho(x_0)^\alpha-s^\alpha\right)}{|B|^{1+\frac{\alpha}{n}}}\int_B|\psi(x)-\psi_B|~dx \\
     &\le \frac{C_\alpha}{|B|^{1+\frac{\alpha}{n}}}\int_B|\psi(x)f_{x_0,s}(x)-(\psi f_{x_0,s})_B|~dx \\
     &\leq C_\alpha\|\psi f_{x_0,s}\|_{BMO_\L^\alpha}\leq C_\alpha\|f_{x_0,s}\|_{BMO_\L^\alpha}\leq C.
\end{align*}
The constants $C$ and $C_\alpha$ appearing in this proof do not depend on $x_0\in\Real^n$ and $0<s\leq\tfrac{1}{2}\rho(x_0)$.

For the converse statement, assume $\psi$ satisfies the properties required in the hypothesis. The kernel of the operator $T=T_\psi$ is zero and $T_\psi1(x)=\psi(x)$, so the conclusion follows by Theorem \ref{Thm:cri BMO alpha}.

The proof of \textit{(A)} is completely analogous by using the function $g_{x_0,s}(x)$ of Lemma \ref{Lem:Examples} instead of $f_{x_0,s}(x)$ and by applying Theorem \ref{Thm:cri BMO}.
\end{proof}

\section{Applications}\label{Section:Applications}

In the following subsections, we prove Theorems \ref{Thm:0-operators} and \ref{Thm:Pola}. In order to adapt  our results to the applications we need the following remark.

\begin{rem}[Vector-valued setting]\label{Rem:Useful}\normalfont
Theorems \ref{Thm:cri BMO} and \ref{Thm:cri BMO alpha} can also be stated in a vector valued setting. If $Tf$ takes values in a Banach space $\mathbb{B}$ and the absolute values in the conditions are replaced by the norm in $\mathbb{B}$ then both results hold.
\end{rem}
\subsection{Maximal operators for the heat--diffusion semigroup $e^{-t\L}$.}\label{Subsection:Maximal heat}

Let $\set{\W_t}_{t>0}$ be the heat--diffusion semigroup associated to $\L$:
$$\W_tf(x)\equiv e^{-t\L}f(x)=\int_{\Real^n}\W_t(x,y)f(y)~dy,\qquad f\in L^2(\Real^n),~x\in\Real^n,~t>0.$$
The kernel of the classical heat semigroup $\set{W_t}_{t>0}=\{e^{t\Delta}\}_{t>0}$ on $\Real^n$ is
$$W_t(x):=\frac{1}{(4\pi t)^{n/2}}~e^{-\frac{\abs{x}^2}{4t}},\qquad x\in\Real^n,~t>0,$$
In the following arguments we need some well known estimates about the kernel $\W_t(x,y)$.

\begin{lem}[See \cite{Dziubanski-Zienkiewicz Hp,Kurata}]\label{Lem:cota heat L}
For every $N>0$ there exists a constant $C_N$ such that
$$0\leq\W_t(x,y)\leq C_Nt^{-n/2}e^{-\frac{\abs{x-y}^2}{5t}}\left(1+\frac{\sqrt{t}}{\rho(x)}
+\frac{\sqrt{t}}{\rho(y)}\right)^{-N},\quad x,y\in\Real^n,~t>0.$$
\end{lem}

\begin{lem}[See {\cite[Proposition~2.16]{Dziubanski-Zienkiewicz Hp}}]\label{Lem:Schwartz}
There exists a nonnegative function $\omega\in\mathcal{S}$ such that
$$\abs{\W_t(x,y)-W_t(x-y)}\leq\left(\frac{\sqrt{t}}{\rho(x)}\right)^{\delta_0}\omega_t(x-y),\quad x,y\in\Real^n,~t>0,$$
where $\omega_t(x-y):=t^{-n/2}\omega\left((x-y)/\sqrt{t}\right)$ and
$$\delta_0:=2-\frac{n}{q}>0.$$
\end{lem}

In fact, going through the proof of \cite{Dziubanski-Zienkiewicz Hp} we see that $\omega(x)=e^{-\abs{x}^2}$.

\begin{lem}[See {\cite[Proposition~4.11]{Dziubanski-Zienkiewicz Colloquium}}]\label{Lem:Regularity heat}
For every $0<\delta<\delta_0,$ there exists a constant $c>0$ such that for every $N>0$ there exists a constant $C>0$ such that for $\abs{y-z}<\sqrt{t}$ we have
$$\abs{\W_t(x,y)-\W_t(x,z)}\le C\left(\frac{\abs{y-z}}{\sqrt{t}}\right)^\delta t^{-n/2}~e^{-c\abs{x-y}^2/t}\left(1+\frac{\sqrt{t}}{\rho(x)}+\frac{\sqrt{t}}{\rho(y)}\right)^{-N}.$$
\end{lem}

\begin{lem}[See {\cite[Proposition~2.17]{Dziubanski-Zienkiewicz Hp}}]\label{Lem:dif of dif}
For every $0<\delta<\min\{1,\delta_0\}$,
$$\abs{\left(\W_t(x,y)-W_t(x-y)\right)-\left(\W_t(x,z)-W_t(x-z)\right)}\leq C\left(\frac{\abs{y-z}}{\rho(x)}\right)^\delta\omega_t(x-y),$$
for all $x,y\in\Real^n$ and $t>0$, with $\abs{y-z}<C\rho(y)$ and $\abs{y-z}<\tfrac{1}{4}\abs{x-y}$.
\end{lem}

To prove that the maximal operator $\W^\ast$ defined by $\W^\ast f(x)=\sup_{t>0}\abs{\W_tf(x)}$ is bounded from $BMO_\L^\alpha$ into itself we give a vector-valued interpretation of the operator and apply Remark \ref{Rem:Useful}. Indeed, it is clear that $\W^\ast f=\norm{\W_tf}_E$, with $E=L^\infty((0,\infty),dt)$. Hence, it is enough to show that the operator $\Lambda(f):=(\W_tf)_{t>0}$ is bounded from $BMO^\alpha_{\L}$ into $BMO^\alpha_{\L,E}$, where the space $BMO_{\L,E}^\alpha$ is defined in the obvious way by replacing the absolute values $|\cdot|$ by norms $\|\cdot\|_E$.

By the Spectral Theorem, $V$ is bounded from $L^2(\Real^n)$ into $L^2_E(\Real^n)$. The desired result is then deduced from the following proposition.

\begin{prop}\label{Prop:Maximal}
Let $x,y,z\in\Real^n$ and $N>0$. Then
\begin{enumerate}[(i)]
    \item $\displaystyle\norm{\W_t(x,y)}_E\leq\frac{C}{|x-y|^n}~\left(1+\frac{\abs{x-y}}{\rho(x)}+\frac{\abs{x-y}}{\rho(y)}\right)^{-N}$;
    \item $\displaystyle \norm{\W_t(x,y)-\W_t(x,z)}_E+\norm{\W_t(y,x)-\W_t(z,x)}_E\leq {C_\delta}\frac{|y-z|^\delta}{|x-y|^{n+\delta}}$, whenever $|x-y|>2|y-z|$, for {any $0<\delta<2-\frac{n}{q}$};
    \item there exists a constant $C$ such that for every ball $B=B(x,s)$ with $0<s\leq\tfrac{1}{2}\rho(x)$,
     $$\log\left(\frac{\rho(x)}{s}\right)\frac{1}{|B|}\int_B\norm{\W_t1(y)-\left(\W_t1\right)_B}_{E}~dy\leq C,$$
    and, if $\alpha<\min\{1,2-\frac{n}{q}\}$ then
     $$\left(\frac{\rho(x)}{s}\right)^\alpha\frac{1}{|B|}\int_B\norm{\W_t1(y)-\left(\W_t1\right)_B}_{E}~dy\leq C.$$
\end{enumerate}
\end{prop}

\begin{proof}
Let us begin with \textit{(i)}. If $t>\abs{x-y}^2$ then the conclusion is immediate from the estimate of Lemma \ref{Lem:cota heat L}. Assume that $t\leq\abs{x-y}^2$. Then
\begin{align*}
    0&\leq \W_t(x,y) \le \frac{C}{\abs{x-y}^n}~e^{-c\frac{\abs{x-y}^2}{t}}\left(1+\frac{\sqrt{t}}{\rho(x)}+\frac{\sqrt{t}}{\rho(y)}\right)^{-N} \\
     &= \frac{C}{\abs{x-y}^n}~e^{-c\frac{\abs{x-y}^2}{t}}\left(\frac{\sqrt{t}}{\abs{x-y}}\right)^{-N} \left(\frac{\abs{x-y}}{\sqrt{t}}+\frac{\abs{x-y}}{\rho(x)}+\frac{\abs{x-y}}{\rho(y)}\right)^{-N} \\
     &\le \frac{C}{\abs{x-y}^n}~e^{-c\frac{\abs{x-y}^2}{t}}\left(\frac{\sqrt{t}}{\abs{x-y}}\right)^{-N}\left(1+\frac{\abs{x-y}}{\rho(x)} +\frac{\abs{x-y}}{\rho(y)}\right)^{-N} \\
     &\le \frac{C}{\abs{x-y}^n}\left(1+\frac{\abs{x-y}}{\rho(x)}+\frac{\abs{x-y}}{\rho(y)}\right)^{-N}.
\end{align*}

We prove \textit{(ii)}. Observe that if $|x-y|>2|y-z|$ then $\abs{x-y}\sim\abs{x-z}.$  For any $0<\delta<\delta_0$, if $\abs{y-z}\le \sqrt{t}$, by Lemma \ref{Lem:Regularity heat},
\begin{equation}\label{difW_t}
\abs{\W_t(x,y)-\W_t(x,z)}\leq C\left(\frac{\abs{y-z}}{\sqrt{t}}\right)^\delta t^{-n/2}e^{-c\frac{\abs{x-y}^2}{t}}\le C\frac{\abs{y-z}^\delta}{\abs{x-y}^{n+\delta}}.
\end{equation}
Consider the situation $\abs{y-z}>\sqrt{t}$. Then Lemma
\ref{Lem:cota heat L} gives
$$\abs{\W_t(x,y)}\le C\left(\frac{\abs{y-z}}{\sqrt{t}}\right)^\delta t^{-n/2}e^{-c\frac{\abs{x-y}^2}{t}}\left(1+\frac{\sqrt{t}}{\rho(x)}+\frac{\sqrt{t}}{\rho(y)}\right)^{-N}\le C\frac{\abs{y-z}^\delta}{\abs{x-y}^{n+\delta}}.$$
The same bound is valid for $\W_t(x,z)$ because
$\abs{x-z}\sim\abs{x-y}$. Then the estimate follows directly since
$\abs{\W_t(x,y)-\W_t(x,z)}\le\abs{\W_t(x,y)}+\abs{\W_t(x,z)}$. The
symmetry of the kernel $\W_t(x,y)=\W_t(y,x)$ gives the conclusion of
\textit{(ii)}.

Let us prove the first statement of \textit{(iii)}. Let $B=B(x,s)$ with $0<s\leq\tfrac{1}{2}\rho(x)$. The triangle inequality gives
\begin{equation}\label{eq7}
\norm{\W_t1(y)-\left(\W_t1\right)_B}_{E}\le
\frac{1}{\abs{B}}\int_B\norm{\W_t1(y)-\W_t1(z)}_E~dz
\end{equation}
We estimate the integrand $\norm{\W_t1(y)-\W_t1(z)}_E$. Because
$y,z\in B$, we have $\rho(y)\sim \rho(z)\sim \rho(x)$ (see Lemma
\ref{Lem:equiv rho}). The fact that $W_t1(x)\equiv1$ and Lemma
\ref{Lem:Schwartz} entails
\begin{align}\label{eq4ch}
    |\W_t1&(y)-\W_t1(z)| \le \abs{\W_t1(y)-W_t1(y)}+\abs{\W_t1(z)-W_t1(z)}\nonumber \\
     &\le \int_{\Real^n}\left[\left(\frac{\sqrt{t}}{\rho(y)}\right)^{\delta_0}\omega_t(y-w)+\left(\frac{\sqrt{t}}{\rho(z)}\right)^{\delta_0} \omega_t(z-w)\right]~dw \nonumber \\
     &\le \left(\frac{\sqrt{t}}{\rho(x)}\right)^{\delta_0}\int_{\Real^n}\left[\omega_t(y-w)+\omega_t(z-w)\right]~dw=C \left(\frac{\sqrt{t}}{\rho(x)}\right)^{\delta_0}.
\end{align}
So \eqref{eq4ch} gives
\begin{equation}\label{eq4}
\abs{\W_t1(y)-\W_t1(z)}\leq C\left(\frac{s}{\rho(x)}\right)^{\delta_0},\quad\hbox{when}~\sqrt{t}\le 2s.
\end{equation}
If $\sqrt{t}>2s$ then $\abs{y-z}\le 2s<\sqrt{t}$. Hence Lemma \ref{Lem:Regularity heat} implies that
\begin{equation}\label{eq5ch}
\begin{aligned}
\abs{\W_t1(y)-\W_t1(z)}&\le\int_{\Real^n}\abs{\W_t(y,w)-\W_t(z,w)}~dw \\
&\le C\left(\frac{\abs{y-z}}{\sqrt{t}}\right)^\delta\le C\left(\frac{s}{\sqrt{t}}\right)^\delta,
\end{aligned}
\end{equation}
where $0<\delta<\delta_0$. Therefore estimate \eqref{eq5ch} gives
\begin{equation}\label{eq5}
\abs{\W_t1(y)-\W_t1(z)}\le C\left(\frac{s}{\rho(x)}\right)^\delta,\quad\hbox{when}~\sqrt{t}>\rho(x).
\end{equation}
When $2s<\sqrt{t}<\rho(x)$ we write
\begin{align*}
    \abs{\W_t1(y)-\W_t1(z)} &= \abs{\left(\W_t1(y)-W_t1(y)\right)-\left(\W_t1(z)-W_t1(z)\right)} \\
     &= \Big|\Big(\int_{\abs{w-y}>C\rho(y)}+\int_{4\abs{y-z}<\abs{w-y}<C\rho(y)}+\int_{\abs{w-y}<4\abs{y-z}}\Big) \\
     &\qquad \left(\W_t(y,w)-W_t(y,w)\right)-\left(\W_t(z,w)-W_t(z,w)\right)~dw\Big| \\
     &= \abs{I+II+III}.
\end{align*}
For $I$ we use the smoothness proved in part \textit{(ii)} of this proposition. Note that the same smoothness estimate is valid for the classical heat kernel. So we get
$$\abs{I}\leq C\int_{\abs{w-y}>C\rho(y)}\frac{\abs{y-z}^\delta}{\abs{w-y}^{n+\delta}}~dw\leq C\left(\frac{s}{\rho(x)}\right)^\delta.$$
In $II$ we apply Lemma \ref{Lem:dif of dif} and the fact that $\rho(w)\sim\rho(y)$ in the region of integration:
$$\abs{II}\leq C\abs{y-z}^\delta\int_{C\rho(y)>\abs{w-y}>4\abs{y-z}}\frac{\omega_t(w-y)}{\rho(w)^\delta}~dw\leq C\left(\frac{s}{\rho(x)}\right)^\delta.$$
The estimate of $III$ is obtained by applying Lemma
\ref{Lem:Schwartz}:
\begin{align*}
    \abs{III} &\leq C\left(\frac{\sqrt{t}}{\rho(x)}\right)^{\delta_0}\left(\int_{\abs{w-y}<4\abs{y-z}}\omega_t(y-w)dw+ \int_{\abs{w-z}\leq5\abs{y-z}}\omega_t(z-w)dw\right) \\
     &\leq C\left(\frac{\sqrt{t}}{\rho(x)}\right)^{\delta_0}\int_{\abs{\xi}\leq5\frac{\abs{y-z}}{\sqrt{t}}}\omega(\xi)~d\xi\leq C\left(\frac{\sqrt{t}}{\rho(x)}\right)^{\delta_0}\left(\frac{\abs{y-z}}{\sqrt{t}}\right)^n \\
     &\leq C\frac{s^n}{\rho(x)^{\delta_0}(\sqrt{t})^{n-\delta_0}}\leq C\frac{s^n}{\rho(x)^{\delta_0}s^{n-\delta_0}}=C\left(\frac{s}{\rho(x)}\right)^{\delta_0},
\end{align*}
since $2s<\sqrt{t}$ and $n-\delta_0>0$. Thus
\begin{align}\label{eq3}
\abs{\W_t1(y)-\W_t1(z)}\le C\left(\frac{s}{\rho(x)}\right)^\delta,\quad\hbox{when}~2s<\sqrt{t}<\rho(x).
\end{align}
Combining \eqref{eq4}, \eqref{eq5} and \eqref{eq3}, we get
\begin{equation}\label{eq6}
\norm{\W_t1(y)-\W_t1(z)}_E\le C\left(\frac{s}{\rho(x)}\right)^{\delta}.
\end{equation}
Therefore, from \eqref{eq7} and \eqref{eq6} we get
$$\log\left(\frac{\rho(x)}{s}\right)\frac{1}{|B|}\int_B\norm{\W_t1(y)-\left(\W_t1\right)_B}_{E}~dy\le C\left(\frac{s}{\rho(x)}\right)^\delta\log\left(\frac{\rho(x)}{s}\right)\le C,$$
which is the first conclusion of \textit{(iii)}.

For the second estimate of \textit{(iii)}, by \eqref{eq6}, we have
$$\left(\frac{\rho(x)}{s}\right)^\alpha\frac{1}{|B|}\int_B\norm{\W_t1(y)-\left(\W_t1\right)_B}_{E}~dy\le C\left(\frac{s}{\rho(x)}\right)^{\delta-\alpha}\le C,$$
as soon as $\delta-\alpha\geq0$, which can be guaranteed if $\alpha<\min\{1,2-\frac{n}{q}\}$ {and we choose $\delta\geq\alpha$.}
\end{proof}

\subsection{Maximal operators for the generalized Poisson operators $\P^{\sigma}_t$.}

For $0<\sigma<1$ we define the generalized Poisson operators $\P^\sigma_t$ as
\begin{equation}\label{generalized Poisson}
\begin{aligned}
u(x,t)\equiv\P^\sigma_tf(x)&=\frac{t^{2\sigma}}{4^\sigma\Gamma(\sigma)}\int_0^\infty e^{-\frac{t^2}{4r}}\W_rf(x)~\frac{dr}{r^{1+\sigma}}\\
&=\frac{1}{\Gamma(\sigma)}\int_0^\infty e^{-r}\W_{\frac{t^2}{4r}}f(x)~\frac{dr}{r^{1-\sigma}},
\end{aligned}
\end{equation}
for $x\in\Real^n$ and $t>0$. The function $u$ satisfies the following boundary value (extension) problem:
$$
\left\{
  \begin{array}{ll}
    -\L_xu+\frac{1-2\sigma}{t}u_t+u_{tt}=0,&\hbox{in}~\Real^n\times(0,\infty); \\
    u(x,0)=f(x), & \hbox{on}~\Real^n.
  \end{array}
\right.
$$
{Moreover, $u$ is useful to characterize the fractional powers of $\L$ since
$$-t^{1-2\sigma}u_t(x,t)\big|_{t=0}=c_\sigma\L^\sigma f(x),$$
for some constant $c_\sigma>0$, see \cite{Stinga-Torrea-CPDE}}. The fractional powers $\L^\sigma$ can be defined in a spectral way. When $\sigma=1/2$ we get that $\P^{1/2}_t=e^{-t\L^{1/2}}$ is the classical Poisson semigroup generated by $\L$ given by Bochner's subordination formula, see \cite{SteinTopics}. It follows that
$$\P^\sigma_tf(x)=\int_{\Real^n}\P^\sigma_t(x,y)f(y)~dy,$$
where
\begin{equation}\label{Poisson kernel}
\begin{aligned}
\P_t^\sigma(x,y)&=\frac{t^{2\sigma}}{4^\sigma\Gamma(\sigma)}\int_0^\infty e^{-\frac{t^2}{4r}}\W_r(x,y)~\frac{dr}{r^{1+\sigma}}\\
&=\frac{1}{\Gamma(\sigma)}\int_0^\infty e^{-r}\W_{\frac{t^2}{4r}}(x,y)~\frac{dr}{r^{1-\sigma}}.
\end{aligned}
\end{equation}

To get the boundedness of the maximal operator
$$\P^{\sigma,\ast}f(x):=\sup_{t>0}\abs{\P^\sigma_tf(x)}=\norm{\P^\sigma_tf(x)}_E$$
in $BMO^\alpha_\L$, we proceed using the vector-valued approach and
the boundedness of the maximal heat semigroup $\W^\ast f$. The following proposition completely analogous to
Proposition \ref{Prop:Maximal} holds.

\begin{prop}
The estimates of Proposition \ref{Prop:Maximal} are valid when $\W_t$ is replaced by $\P^\sigma_t$.
\end{prop}

\begin{proof}
The proof follows by transferring the estimates for $\W_t(x,y)$ to $\P^\sigma_t(x,y)$ through formula \eqref{Poisson kernel}. We just sketch the proof of \textit{(iii)}. For any $y,z\in B=B(x,s)$, $x\in\Real^n$, $0<s\le\tfrac{1}{2}\rho(x)$, by \eqref{Poisson kernel}, Minkowski's integral inequality and \eqref{eq6} we have
\begin{align*}
    \norm{\P^\sigma_t1(y)-\P^\sigma_t1(z)}_E&\leq C_\sigma\int_0^\infty t^{2\sigma}e^{-\frac{t^2}{4r}}\norm{\W_r1(y)-\W_r1(z)}_E\frac{dr}{r^{1+\sigma}} \\
     &\le C\left(\frac{s}{\rho(x)}\right)^\delta\int_0^\infty t^{2\sigma}e^{-\frac{t^2}{4r}}\frac{dr}{r^{1+\sigma}}=C\left(\frac{s}{\rho(x)}\right)^\delta.
\end{align*}
Then the same computations for the heat semigroup apply in this case and give \textit{(iii)}.
\end{proof}

\subsection{Littlewood--Paley $g$-function for the heat--diffusion semigroup}

The Littlewood--Paley $g$-function associated with $\set{\W_t}_{t>0}$ is defined by
$$g_{\W}(f)(x)=\left(\int_0^\infty \left|t\partial_t\W_tf(x)\right|^2 \frac{dt}{t}\right)^{1/2}=\|t\partial_t\W_tf(x)\|_F,$$
where $F:=L^2\big((0,\infty),\frac{dt}{t}\big)$. The Spectral Theorem implies that $g_\W$ is an isometry on $L^2(\Real^n)$, see \cite[Lemma~3]{DGMTZ}. As before, to get the boundedness of $g_\W$ from $BMO_\L^\alpha$ into itself it is sufficient to prove the following result.

\begin{prop}\label{gfunction for heat}
The estimates of Proposition \ref{Prop:Maximal} are valid when $\W_t$ is replaced by $t\partial_t\W_t$ and the Banach space $E$ is replaced by $F$.
\end{prop}

The proof of Proposition \ref{gfunction for heat} requires some extra
effort. Let us recall the following already well-known estimates.

\begin{lem}[See {\cite[Proposition~4]{DGMTZ}}]\label{Prop:Heat est}
For any $N>0$ there exist constants $C=C_N$ and $c>0$ such that for all $x,y\in\Real^n$, $t>0$ and $0<\delta<\delta_0$,
\begin{enumerate}[(a)]
    \item $\displaystyle|t\partial_t\W_t(x,y)|\leq Ct^{-n/2}e^{-c\frac{\abs{x-y}^2}{t}}\left(1+\frac{\sqrt{t}}{\rho(x)}+\frac{\sqrt{t}}{\rho(y)}\right)^{-N}$;
    \item For all $\abs{h}\leq\sqrt{t}$ we have
    $$|t\partial_t\W_t(x+h,y)-t\partial_t\W(x,y)|\leq C\left(\frac{\abs{h}}{\sqrt{t}}\right)^\delta \frac{e^{-c\frac{\abs{x-y}^2}{t}}}{t^{n/2}}\left(1+\frac{\sqrt{t}}{\rho(x)}+ \frac{\sqrt{t}}{\rho(y)}\right)^{-N},$$
    \item $\displaystyle\abs{\int_{\Real^n}t\partial_t\W_t(x,y)~dy}\leq C\frac{(\sqrt{t}/\rho(x))^\delta}{\left(1+\sqrt{t}/\rho(x)\right)^N}$.
\end{enumerate}
\end{lem}

\begin{proof}[Proof of Proposition \ref{gfunction for heat}]
Part \textit{(i)} is proved using Lemma \ref{Prop:Heat est}\textit{(a)} and the same argument of the proof of Proposition \ref{Prop:Maximal}\textit{(i)}.

Similarly \textit{(ii)} follows by Lemma \ref{Prop:Heat est}\textit{(b)} and the symmetry $\W_t(x,y)=\W_t(y,x)$.

To prove \textit{(iii)} let us fix $y,z\in B=B(x_0,s)$, $0<s\le\tfrac{1}{2}\rho(x_0)$. In view of an estimate like \eqref{eq7}, we must handle $\norm{t\partial_t\W_t1(y)-t\partial_t\W_t1(z)}_F$ first. We can write
\begin{equation}\label{A}
\begin{aligned}
\|&t\partial_t\W_t1(y)-t\partial_t\W_t1(z)\|_F^2 \\
&= \int_0^\infty \abs{\int_{\Real^n}\left(t\partial_t\W_t(x,y)-t\partial_t\W_t(x,z)\right)~dx}^2\frac{dt}{t} \\
&=\left(\int_0^{4s^2}+\int_{4s^2}^{\rho(x_0)^2}+\int_{\rho(x_0)^2}^\infty\right) \abs{\int_{\Real^n}\left(t\partial_t\W_t(x,y)-t\partial_t\W_t(x,z)\right)dx}^2\frac{dt}{t} \\
&=:A_1+A_2+A_3.
\end{aligned}
\end{equation}
Since $y,z\in B\subset B(x_0,\rho(x_0))$, it follows that $\rho(y)\sim\rho(x_0)\sim\rho(z)$. By Lemma \ref{Prop:Heat est}\textit{(c)},
\begin{equation}\label{A1}
\begin{aligned}
A_1&\le C\int_0^{4s^2}\frac{(\sqrt{t}/\rho(x_0))^{2\delta}}{(1+\sqrt{t}/\rho(x_0))^{2N}}~
\frac{dt}{t} \\
&\le C\int_0^{4s^2}\left(\frac{\sqrt{t}}{\rho(x_0)}\right)^{2\delta}~\frac{dt}{t} =C\left(\frac{s}{\rho(x_0)}\right)^{2\delta}.
\end{aligned}
\end{equation}
Also, by Lemma \ref{Prop:Heat est}\textit{(b)},
\begin{equation}\label{A3}
\begin{aligned}
    A_3 &\le C \int_{\rho(x_0)^2}^\infty \left(\frac{\abs{y-z}}{\sqrt t}\right)^{2\delta}\abs{\int_{\Real^n}t^{-n/2}e^{-c\frac{\abs{x-y}^2}{t}}~dx}^2\frac{dt}{t}  \\
     &= C\int_{\rho(x_0)^2}^\infty\left(\frac{\abs{y-z}}{\sqrt t}\right)^{2\delta}\frac{dt}{t}\le C\left(\frac{s}{\rho(x_0)}\right)^{2\delta}.
\end{aligned}
\end{equation}
It remains to estimate the term $A_2$. Recall from \cite[Eq.~(2.8)]{DGMTZ} that, {because the potential $V$ is in the reverse H\"older class},
\begin{equation}\label{Dziubanski 2}
\int_{\Real^n}\omega_t(x-y)V(y)~dy\leq\frac{C}{t}\left(\frac{\sqrt{t}}{\rho(x)}\right)^\delta,\quad\hbox{for}~t\leq\rho(x)^2.
\end{equation}
{Clearly $\partial_t\W_t1(x)=\L\W_t1(x)=\W_tV(x)$, that is}
\begin{equation}\label{Dziubanski 1}
\int_{\Real^n}\partial_t\W_t(x,y)~dy=\int_{\Real^n}\W_t(x,y)V(y)~dy.
\end{equation}
We then have, by Lemma \ref{Lem:Regularity heat} (remember that $\abs{y-z}\leq2s\leq\sqrt{t}$),
\begin{equation}\label{A2}
\begin{aligned}
    A_2 &= \int_{4s^2}^{\rho(x_0)^2}\abs{\int_{\Real^n}\left(t\partial_t\W_t(x,y)-t\partial_t\W_t(x,z)\right)~dx~}^2\frac{dt}{t} \\
     &= \int_{4s^2}^{\rho(x_0)^2}t\abs{\int_{\Real^n}\left(\W_t(y,x)-\W_t(z,x)\right)V(x)~dx~}^2dt \\
     &\leq C\abs{y-z}^{2\delta}\int_{4s^2}^{\rho(x_0)^2}t^{1-\delta}\abs{\int_{\Real^n}t^{-n/2}e^{-c\frac{\abs{y-x}}{t}}V(x)~dx~}^2dt \\
     &\leq Cs^{2\delta}\int_{4s^2}^{\rho(x_0)^2}t^{1-\delta}t^{-2}\left(\frac{\sqrt{t}}{\rho(y)}\right)^{2\delta}dt \\
     &\le C\left(\frac{s}{\rho(x_0)}\right)^{2\delta}\int_{s^2}^{\rho(x_0)^2}\frac{dt}{t}= C\left(\frac{s}{\rho(x_0)}\right)^{2\delta}\log\left(\frac{\rho(x_0)}{s}\right).
\end{aligned}
\end{equation}
Combining \eqref{A}, \eqref{A1}, \eqref{A3} and \eqref{A2} we get
\begin{equation}\label{norm of K}
\norm{t\partial_t\W_t1(y)-t\partial_t\W_t1(z)}_F\le
C\left(\frac{s}{\rho(x_0)}\right)^{\delta}\left(\log\left(\frac{\rho(x_0)}{s}\right)\right)^{1/2}.
\end{equation}
Thus \textit{(iii)} readily follows.
\end{proof}

\subsection{Littlewood--Paley $g$-function for the Poisson semigroup}

The Littlewood--Paley $g$-function associated with the Poisson semigroup $\set{\P_t}_{t>0}$ $\equiv$ $\{\P^{1/2}_t\}_{t>0}$ (see \eqref{generalized Poisson} and \eqref{Poisson kernel}) is defined analogously as $g_\W$ by replacing the heat semigroup by the Poisson semigroup:
$$g_{\P}(f)(x)=\left(\int_0^\infty \left|t\partial_t\P_tf(x)\right|^2 \frac{dt}{t}\right)^{1/2}=\|t\partial_t\P_tf(x)\|_F.$$
By Spectral Theorem, $g_\P$ is an isometry on $L^2(\Real^n)$, see \cite[Lemma~3.7]{Ma-Stinga-Torrea-Zhang}. We also have

\begin{prop}\label{gfunction}
The estimates of Proposition \ref{Prop:Maximal} are valid when $\W_t$ is replaced by $t\partial_t\P_t$ and the Banach space $E$ is replaced by $F$.
\end{prop}

\begin{proof}
First we derive a convenient formula to treat the operator
$t\partial_t\P_t$. By the second identity of \eqref{Poisson kernel}
with $\sigma=1/2$ (Bochner's subordination formula) and a change of
variables,
\begin{align}\label{derivative formula}
    t\partial_t\P_t(x,y) &= \frac{t}{\sqrt{\pi}}\int_0^\infty\frac{e^{-r}}{r^{1/2}}~\partial_t\left(\W_{\frac{t^2}{4r}}(x,y)\right)~dr \nonumber\\
     &= \frac{t^2}{2\sqrt{\pi}}\int_0^\infty\frac{e^{-r}}{r^{1/2}}~\partial_v\left(\W_v(x,y)\right)\Big|_{v=\frac{t^2}{4r}}~\frac{dr}{r} \\
     &= \frac{t}{\sqrt{\pi}}\int_0^\infty e^{-\frac{t^2}{4v}}~v\partial_v\W_v(x,y)~\frac{dv}{v^{3/2}}.\nonumber
\end{align}
Formula \eqref{derivative formula} should be compared with the first
identity of \eqref{Poisson kernel} for $\sigma=1/2$. It will allow
us to transfer the estimates for $v\partial_v\W_v$ to
$t\partial_t\P_t$.

For \textit{(i)} we use \eqref{derivative formula}, Minkowski's
integral inequality and the estimate for $v\partial_v\W_v$:
\begin{align*}
    \norm{t\partial_t\P_t(x,y)}_F^2 &\leq C\int_0^\infty\abs{v\partial_v\W_v(x,y)}^2\int_0^\infty   te^{-\frac{t^2}{4v}}~\frac{dt}{t}~\frac{dv}{v^{3/2}} \\
&= C\int_0^\infty\abs{v\partial_v\W_v(x,y)}^2~\frac{dv}{v} \\
     &\le\frac{C}{|x-y|^{2n}}\left(1+\frac{\abs{x-y}}{\rho(x)} +\frac{\abs{x-y}}{\rho(y)}\right)^{-2N}.
\end{align*}

The estimate for \textit{(ii)} follows in the same way.

By \eqref{derivative formula}, Fubini's Theorem and \eqref{norm of K},
$$\norm{t\partial_t\P_t1(y)-t\partial_t\P_t1(z)}_F\leq C\left(\frac{s}{\rho(x_0)}\right)^{\delta}\log\left(\frac{\rho(x_0)}{s}\right)^{1/2}.$$
which is sufficient for \textit{(iii)}.
\end{proof}

\subsection{Laplace transform type multipliers}\label{Subsection:Laplace}

Given a bounded function $a$ on $[0,\infty)$ we let
$$m(\lambda)=\lambda\int_0^\infty a(t)e^{-t\lambda}~dt.$$
The Spectral Theorem allows us to define the Laplace transform type multiplier operator $m(\L)$ associated to $a$ that is bounded on $L^2(\Real^n)$. Observe that
$$m(\L)f(x)=\int_0^\infty a(t)\L e^{-t\L}f(x)~dt=\int_0^\infty a(t)\partial_t\W_tf(x)~dt,\quad x\in\Real^n.$$
Then the kernel $\M(x,y)$ of $m(\L)$ can be written as
$$\M(x,y)=\int_0^\infty a(t)\partial_t\W_t(x,y)~dt.$$

\begin{prop}
Let $x,y,z\in\Real^n$, $N>0$, $0\leq\alpha<1$ and $B=B(x,s)$ for
$0<s\leq\rho(x)$. Then
\begin{enumerate}[(a)]
    \item $\displaystyle \abs{\M(x,y)}\leq\frac{C}{\abs{x-y}^n}~\left(1+\frac{\abs{x-y}}{\rho(x)}+\frac{\abs{x-y}}{\rho(y)}\right)^{-N}$:
    \item $\displaystyle \abs{\M(x,y)-\M(x,z)}+\abs{\M(y,x)-\M(z,x)}\leq {C_\delta}\frac{\abs{y-z}^\delta}{\abs{x-y}^{n+\delta}}$, for all $\abs{x-y}>2\abs{y-z}$ and any $0<\delta<\delta_0$;
    \item $\displaystyle \log\left(\frac{\rho(x)}{s}\right)\frac{1}{\abs{B}}\int_B\abs{m(\L)1(y)-(m(\L)1)_B}~dy\leq C$;
    \item $\displaystyle \left(\frac{\rho(x)}{s}\right)^\alpha\frac{1}{\abs{B}}\int_B\abs{m(\L)1(y)-(m(\L)1)_B}~dy
    \leq C$, for any {$0\le\alpha<\min\{1,2-\frac{n}{q}\}$}.
\end{enumerate}
\end{prop}

\begin{proof}
The reader should recall the estimates for $\partial_t\W_t(x,y)$ stated in Lemma \ref{Prop:Heat est}.

For \textit{(a)}, by Lemma \ref{Prop:Heat est}\textit{(a)},
\begin{align*}
    &\int_0^{\abs{x-y}^2}\abs{a(t)\partial_t\W_t(x,y)}~dt \\
     &\leq C\int_0^{\abs{x-y}^2}t^{-n/2}e^{-c\frac{\abs{x-y}^2}{t}} \left(1+\frac{\sqrt{t}}{\rho(x)}+\frac{\sqrt{t}}{\rho(y)}\right)^{-N}\frac{dt}{t} \\
     &= C\int_0^{\abs{x-y}^2}t^{-n/2}e^{-c\frac{\abs{x-y}^2}{t}}\left(\frac{\sqrt{t}}{\abs{x-y}}\right)^{-N} \left(\frac{\abs{x-y}}{\sqrt{t}}+\frac{\abs{x-y}}{\rho(x)}+\frac{\abs{x-y}}{\rho(y)}\right)^{-N}\frac{dt}{t} \\
     &\leq C\int_0^{\abs{x-y}^2}t^{-n/2}e^{-c\frac{\abs{x-y}^2}{t}}\left(1+\frac{\abs{x-y}}{\rho(x)}+\frac{\abs{x-y}}{\rho(y)}\right)^{-N}\frac{dt}{t} \\
     &\le \frac{C}{\abs{x-y}^n}\left(1+\frac{\abs{x-y}}{\rho(x)}+\frac{\abs{x-y}}{\rho(y)}\right)^{-N},
\end{align*}
and
\begin{align*}
    \int_{\abs{x-y}^2}^\infty&\abs{a(t)\partial_t\W_t(x,y)}~dt \leq C\int_{\abs{x-y}^2}^\infty t^{-n/2}e^{-c\frac{\abs{x-y}^2}{t}} \left(1+\frac{\sqrt{t}}{\rho(x)}+\frac{\sqrt{t}}{\rho(y)}\right)^{-N}\frac{dt}{t} \\
    &\leq C\int_{\abs{x-y}^2}^\infty t^{-n/2}e^{-c\frac{\abs{x-y}^2}{t}}\left(1+\frac{\abs{x-y}}{\rho(x)}+\frac{\abs{x-y}}{\rho(y)}\right)^{-N}\frac{dt}{t} \\
     &\le \frac{C}{\abs{x-y}^n}\left(1+\frac{\abs{x-y}}{\rho(x)}+\frac{\abs{x-y}}{\rho(y)}\right)^{-N}.
\end{align*}

To check \textit{(b)} we apply Lemma \ref{Prop:Heat est}\textit{(b)} to see that
\begin{align*}
    \int_{\abs{x-y}^2}^\infty\abs{a(t)}&\abs{\partial_t\W_t(x,y)-\partial_t\W_t(x,z)}~dt \\
     &\leq C\int_{\abs{x-y}^2}^\infty\left(\frac{\abs{y-z}}{\sqrt{t}}\right)^\delta t^{-n/2}e^{-c\frac{\abs{x-y}^2}{t}} \frac{dt}{t} \\
    &\leq C\frac{\abs{y-z}^\delta}{\abs{x-y}^{n+\delta}}.
\end{align*}
Moreover, by Lemma \ref{Prop:Heat est}\textit{(a)},
\begin{align*}
\int_0^{\abs{x-y}^2}\abs{a(t)\partial_t\W_t(x,y)}~dt &\leq C\int_0^{\abs{x-y}^2}\left(\frac{\abs{y-z}}{\sqrt{t}}\right)^\delta t^{-n/2}e^{-c\frac{\abs{x-y}^2}{t}}~\frac{dt}{t} \\
&\leq C\frac{\abs{y-z}^\delta}{\abs{x-y}^{n+\delta}}.
\end{align*}
The same bound is valid for
$\int_0^{\abs{x-y}^2}\abs{a(t)}\abs{\partial_t\W_t(x,z)}\frac{dt}{t}$
because $\abs{x-z}\sim\abs{x-y}$. The symmetry of the kernel
$\M(x,y)=\M(y,x)$ gives the conclusion of \textit{(b)}.

Fix $y,z\in B$. For \textit{(c)} and \textit{(d)}, let us estimate the difference
$$\abs{m(\L)1(y)-m(\L)1(z)}\leq \norm{a}_{L^\infty}\int_0^\infty\abs{\int_{\Real^n}\left(\partial_t\W_t(y,w)-\partial_t\W_t(z,w)\right)dw}dt.$$
To that end we split the integral in $t$ into three parts. We start with the part from $0$ to $4s^2$.
 From Lemma \ref{Prop:Heat est}\textit{(c)},
$$\abs{\int_0^{4s^2}\int_{\Real^n}\left(\partial_t\W_t(y,w)-\partial_t\W_t(z,w)\right)dw\,dt}\leq C\int_0^{4s^2}\left(\frac{\sqrt{t}}{\rho(x)}\right)^{\delta}\frac{dt}{t}=C\left(\frac{s}{\rho(x)}\right)^{\delta}.$$
Let us continue with the integral from $\rho(x)^2$ to $\infty$. We apply Lemma \ref{Prop:Heat est}\textit{(b)}:
\begin{align*}
    \abs{\int_{\rho(x)^2}^\infty\int_{\Real^n}\left(\partial_t\W_t(y,w)-\partial_t\W_t(z,w)\right)~dw~dt} &\leq C\int_{\rho(x)^2}^\infty\left(\frac{\abs{y-z}}{\sqrt{t}}\right)^{\delta}~\frac{dt}{t} \\
     &\leq C\left(\frac{s}{\rho(x)}\right)^{\delta}.
\end{align*}
{Finally we consider the part from $4s^2$ to $\rho(x)^2$. Applying \eqref{Dziubanski 1}, Lemma \ref{Lem:Regularity heat} and \eqref{Dziubanski 2}},
\begin{align*}
    &\abs{\int_{4s^2}^{\rho(x)^2}\int_{\Real^n}\left(\partial_t\W_t(y,w)-\partial_t\W_t(z,w)\right)~dw~dt} \\
     &= \int_{4s^2}^{\rho(x)^2}\abs{\int_{\Real^n}\left(\W_t(y,w)-\W_t(z,w)\right)V(w)~dw~}dt \\
     &\leq C\abs{y-z}^\delta\int_{4s^2}^{\rho(x)^2}\int_{\Real^n}t^{-n/2}e^{-c\frac{\abs{y-w}^2}{t}}V(w)~dw~\frac{dt}{t^{\delta/2}} \\
     &\leq C\left(\frac{s}{\rho(y)}\right)^\delta\int_{s^2}^{\rho(x)^2}~\frac{dt}{t}\leq C\left(\frac{s}{\rho(x)}\right)^{\delta}\log\left(\frac{\rho(x)}{s}\right).
\end{align*}
Hence
\begin{align*}
    \frac{1}{|B|}\int_{B}\abs{m(\L)1(y)-(m(\L)1)_B}~dy &\leq \frac{C}{s^{2n}} \int_B\int_B|m(\L)1(y)-m(\L)(z)|\,dy\,dz \\
     &\leq C\left(\frac{s}{\rho(x)}\right)^{\delta}\log\left(\frac{\rho(x)}{s}\right).
\end{align*}
Thus \textit{(c)} is valid and also \textit{(d)} holds when
$\alpha<\delta$.
\end{proof}

\subsection{Riesz transforms}

For every $i=1,2,\ldots,n$, the $i$-th Riesz transform $\R_i$ associated to $\L$ is defined by
$$\R_i=\partial_{x_i}\L^{-1/2}=\partial_{x_i}\frac{1}{\sqrt{\pi}}\int_0^\infty e^{-t\L}~\frac{dt}{t^{1/2}}.$$
We denote by $\R$ the vector $\nabla\L^{-1/2}=(\R_1,\ldots,\R_n)$.
The Riesz transforms associated to $\L$ were first studied by Z.
Shen in \cite{Shen}. He showed (Theorem 0.8 of \cite{Shen}) that if
the potential $V\in RH_q$ with $q>n$ then $\R$ is a
Calder\'on--Zygmund operator. In particular, the $\Real^n$--valued
operator $\R$ is bounded from $L^2(\Real^n)$ into
$L^2_{\Real^n}(\Real^n)$ and its kernel $\K$ satisfies, for any
$0<\delta<1-\frac{n}{q}$,
\begin{equation}\label{Riesz ker smooth}
\abs{\K(x,y)-\K(x,z)}+\abs{\K(y,x)-\K(z,x)}\le C\frac{\abs{y-z}^\delta}{\abs{x-y}^{n+\delta}},
\end{equation}
whenever $\abs{x-y}>2\abs{y-z}$. Moreover, when $q>n$ we have that for any $x,y\in\Real^n$, $x\neq y$, and $N>0$ there exists a constant $C_N$ such that
\begin{equation}\label{Riesz ker size}
\abs{\K(x,y)}\le \frac{C_N}{\abs{x-y}^n}\left(1+\frac{\abs{x-y}}{\rho(x)}\right)^{-N},
\end{equation}
see \cite[Eq.~(6.5)]{Shen} and also {\cite[Lemma~3]{Bongioanni-Harboure-Salinas-Riesz}}. Hence $\R$ is a $\gamma$-Schr\"odinger-Calder\'on-Zygmuund operator with $\gamma=0$.

The boundedness results of $\R$ in $BMO^\alpha_\L$ follow by checking the properties of $\R1$.

\begin{prop}\label{Prop:T1 cond R}
Let $V\in RH_q$ with $q>n$ and $B=B(x_0,s)$ for $x_0\in\Real^n$ and $0<s\le\tfrac{1}{2}\rho(x_0)$. Then
\begin{enumerate}[(i)]
    \item $\displaystyle\log\left(\frac{\rho(x_0)}{s}\right)\frac{1}{|B|}\int_B\abs{{\R1(y)}-(\R1)_B}~dy\leq C$;
    \item $\displaystyle\left(\frac{\rho(x_0)}{s}\right)^\alpha\frac{1}{|B|}\int_{B}\abs{\R1(y)-(\R1)_B}~dy\leq C,$ for $\alpha<1-\frac{n}{q}$.
\end{enumerate}
\end{prop}

To prove Proposition \ref{Prop:T1 cond R}, we collect some
well-known estimates on $\K(x,y)$. Let us denote by $\K_0$ the
kernel of the ($\Real^n$--valued) classical Riesz transform
$\R_0=\nabla(-\Delta)^{-1/2}$.

\begin{lem}[{\cite[Lemmas~3~and~4]{Bongioanni-Harboure-Salinas-Riesz}}]\label{Lem:comp Riesz kernel}
Suppose that $V\in RH_q$ with $q>n$.
\begin{enumerate}[(a)]
    \item For any $x,y\in\Real^n$, $x\neq y$,
     $$\abs{\K(x,y)-\K_0(x,y)}\le\frac{C}{\abs{x-y}^n}\left(\frac{\abs{x-y}}{\rho(x)}\right)^{2-n/q}.$$
    \item For any $0<\delta<1-\frac{n}{q}$ there exists a constant $C$ such that if $\abs{z-y}\ge 2\abs{x-y}$ then
     $$\abs{\left(\K(x,z)-\K_0(x,z)\right)-\left(\K(y,z)-\K_0(y,z)\right)}\le C\frac{\abs{x-y}^\delta}{\abs{z-y}^{n+\delta}}\left(\frac{\abs{z-y}}{\rho(z)}\right)^{2-n/q}.$$
\end{enumerate}
\end{lem}

\begin{proof}[Proof of Proposition \ref{Prop:T1 cond R}]
Let $y,z\in B$. Then $\rho(y)\sim\rho(x_0)\sim\rho(z)$. Since
$$\R1(x)=\lim_{\varepsilon\to0^+}\int_{\abs{x-y}>\varepsilon}\K(x,y)~dy,\quad\hbox{a.e.}~x\in\Real^n,$$
we have
\begin{align*}
    \abs{\R1(y)-\R1(z)} &\le \lim_{\varepsilon\to0^+}\abs{\int_{\varepsilon<\abs{x-y}\le4\rho(x_0)}\K(y,x)~dx- \int_{\varepsilon<\abs{x-z}\le4\rho(x_0)}\K(z,x)~dx} \\
     &\quad +\abs{\int_{\abs{x-y}>4\rho(x_0)}\K(y,x)~dx-\int_{\abs{x-z}>4\rho(x_0)}\K(z,x)~dx}\\
     &=:\lim_{\varepsilon\to0^+}A_\varepsilon+B.
\end{align*}

First, let us consider $A_\varepsilon$. Since we will consider the limit as $\varepsilon$ tends to zero, we can assume that $0<\varepsilon<4\rho(x_0)-2s$. For every annulus $E$ we have $\displaystyle \int_E\K_0(x,y)~dy=0$. Therefore,
\begin{equation}
\begin{aligned}\label{A epsilon}
    A_\varepsilon &= \left|\int_{\varepsilon<\abs{x-y}\le4\rho(x_0)}\left(\K(y,x)-\K_0(y,x)\right)~dx\right. \\
     &\quad \left.-\int_{\varepsilon<\abs{x-z}\le4\rho(x_0)}\left(\K(z,x)-\K_0(z,x)\right)~dx\right| \\
     &\leq \abs{\int_{\Real^n}\left(\K(y,x)-\K_0(y,x)\right)\left(\chi_{\varepsilon<\abs{x-y}\le4\rho(x_0)}(x)- \chi_{\varepsilon<\abs{x-z}\le4\rho(x_0)}(x)\right)dx} \\
     &\quad +\abs{\int_{\Real^n}\left[\left(\K(y,x)-\K_0(y,x)\right)-\left(\K(z,x)-\K_0(z,x)\right)\right] \chi_{\varepsilon<\abs{x-z}\le4\rho(x_0)}(x)dx} \\
     &=:A_\varepsilon^1+A_\varepsilon^2.
\end{aligned}
\end{equation}

The term $A_\varepsilon^1$ is not zero when
$\abs{\chi_{\varepsilon<\abs{x-y}\le4\rho(x_0)}(x)-\chi_{\varepsilon<\abs{x-z}\le4\rho(x_0)}(x)}=1$,
namely, when
\begin{itemize}
    \item $\varepsilon<\abs{x-y}\le4\rho(x_0)$ and $\abs{x-z}\le\varepsilon$; or
    \item $\varepsilon<\abs{x-y}\le4\rho(x_0)$ and $\abs{x-z}>4\rho(x_0)$; or
    \item $\varepsilon<\abs{x-z}\le4\rho(x_0)$ and $\abs{x-y}\le\varepsilon$; or
    \item $\varepsilon<\abs{x-z}\le4\rho(x_0)$ and $\abs{x-y}>4\rho(x_0)$.
\end{itemize}
In the first case we have $\varepsilon<\abs{x-y}\le\abs{x-z}+\abs{z-y}<\varepsilon+2s$. Then, by Lemma \ref{Lem:comp Riesz kernel}\textit{(a)},
\begin{equation}\label{A1 first1}
A_\varepsilon^1\le\int_{\varepsilon<\abs{x-y}\le2s+\varepsilon}\frac{C}{\abs{x-y}^n}\left(\frac{\abs{x-y}}{\rho(y)}\right)^{2-n/q}~dx\le
C\left(\frac{s}{\rho(x_0)}\right)^{2-n/q}.
\end{equation}
In the second case, by the assumption on $\varepsilon$, we get $\max\set{\varepsilon,4\rho(x_0)-2s}=4\rho(x_0)-2s<\abs{x-y}\le4\rho(x_0)$. Then Lemma \ref{Lem:comp Riesz kernel}\textit{(a)} and the Mean Value Theorem give
\begin{equation}\label{A1 first2}
A_\varepsilon^1\le \frac{C}{\rho(x_0)^{2-n/q}}\int_{4\rho(x_0)-2s<\abs{x-y}\le4\rho(x_0)}\abs{x-y}^{2-n/q-n}~dx\le C\frac{s}{\rho(x_0)}.
\end{equation}
In the third and fourth cases we obtain the same bounds as in \eqref{A1 first1} and \eqref{A1 first2} by replacing $y$ by $z$. Thus, when $0<\delta<1-n/q$,
\begin{equation}\label{A1 epsilon}
A_\varepsilon^1\le C\left(\frac{s}{\rho(x_0)}\right)^{\delta}.
\end{equation}

We see that $A_\varepsilon^2$ is bounded by $\abs{A_\varepsilon^{2,1}}+\abs{A_\varepsilon^{2,2}}$, where
\begin{equation}
\begin{aligned}\label{A2 epsilon}
   &A_\varepsilon^{2,1}+A_\varepsilon^{2,2}= \\
   &\int_{\abs{x-z}>2\abs{y-z}}\left[\left(\K(y,x)-\K_0(y,x)\right)-\left(\K(z,x)-\K_0(z,x)\right)\right] \chi_{\varepsilon<\abs{x-z}\le4\rho(x_0)}(x)dx \\
    &+\\
    &\int_{\abs{x-z}\le2\abs{y-z}}\left[\left(\K(y,x)-\K_0(y,x)\right)-\left(\K(z,x)-\K_0(z,x)\right)\right] \chi_{\varepsilon<\abs{x-z}\le4\rho(x_0)}(x)dx.
\end{aligned}
\end{equation}
By Lemma \ref{Lem:comp Riesz kernel}\textit{(b)},
\begin{equation}\label{A21 epsilon}
A_\varepsilon^{2,1}\le C\frac{\abs{y-z}^\delta}{\rho(z)^{2-n/q}}\int_{\abs{x-z}\le4\rho(x_0)}\abs{x-z}^{2-n/q-n-\delta}dx\le C\left(\frac{s}{\rho(x_0)}\right)^\delta.
\end{equation}
On the other hand, Lemma \ref{Lem:comp Riesz kernel}\textit{(a)}
gives
\begin{equation}
\begin{aligned}\label{A22 epsilon}
    A_\varepsilon^{2,2} &\le \int_{\abs{x-z}\le2\abs{y-z}}\frac{C}{\abs{x-y}^n}\left(\frac{\abs{x-y}}{\rho(y)}\right)^{2-n/q}dx \\
     &\quad+\int_{\abs{x-z}\le2\abs{y-z}}\frac{C}{\abs{x-z}^n}\left(\frac{\abs{x-z}}{\rho(z)}\right)^{2-n/q}dx \\
     &\le \frac{C}{\rho(x_0)^{2-n/q}}\int_{\abs{x-y}\le 3\abs{y-z}}\abs{x-y}^{2-n/q-n}dx\\
     &\quad +\frac{C}{\rho(x_0)^{2-n/q}}\int_{\abs{x-z}\le2\abs{y-z}}\abs{x-z}^{2-n/q-n}dx \\
     &\le C\left(\frac{s}{\rho(x_0)}\right)^{2-n/q}\le C\left(\frac{s}{\rho(x_0)}\right)^{\delta},
\end{aligned}
\end{equation}
for any $0<\delta<1-n/q$. Hence, from \eqref{A epsilon}, \eqref{A1 epsilon}, \eqref{A2 epsilon}, \eqref{A21 epsilon} and \eqref{A22 epsilon} we obtain that for all $\varepsilon>0$ sufficiently small,
\begin{equation}\label{A est}
A_\varepsilon\le C\left(\frac{s}{\rho(x_0)}\right)^{\delta}.
\end{equation}

Let us now estimate $B$. In a similar way,
\begin{align*}
    B &\le \int_{\abs{x-y}>4\rho(x_0)}\abs{\K(y,x)-\K(z,x)}~dx \\
     &\quad +\int_{\Real^n}\abs{\K(z,x)}\abs{\chi_{\abs{x-z}>4\rho(x_0)}(x)-\chi_{\abs{x-z}>4\rho(x_0)}(x)}~dx\\
     &=:B_1+B_2.
\end{align*}
In the integrand of $B_1$ we have $\abs{x-y}>4\rho(x_0)\ge8s>2\abs{y-z}$. Therefore the smoothness of the Riesz kernel \eqref{Riesz ker smooth} can be applied to get
$$B_1\le C\int_{\abs{x-y}>4\rho(x_0)}\frac{\abs{y-z}^\delta}{\abs{x-y}^{n+\delta}}~dx\le C\left(\frac{s}{\rho(x_0)}\right)^\delta.$$
It is possible to deal with $B_2$ as with $A_\varepsilon^1$ above to
derive the same bound. Hence,
$$B\le C\left(\frac{s}{\rho(x_0)}\right)^\delta.$$
This last estimate together with \eqref{A est} imply
$$\abs{\R1(y)-\R1(z)}\le C\left(\frac{s}{\rho(x_0)}\right)^\delta,$$
where $0<\delta<1-n/q$. From here \textit{(i)} and \textit{(ii)} readily follow.
\end{proof}

\subsection{Negative powers}

For any $\gamma>0$ the negative powers of $\L$ are defined as
$$\L^{-\gamma/2}f(x)=\frac{1}{\Gamma(\gamma/2)}\int_0^\infty e^{-t\L}f(x)~\frac{dt}{t^{1-\gamma/2}}=\int_{\Real^n}\K_\gamma(x,y)f(y)~dy,$$
where
$$\K_\gamma(x,y)=\frac{1}{\Gamma(\gamma/2)}\int_0^\infty \W_t(x,y)~\frac{dt}{t^{1-\gamma/2}},\quad x\in\Real^n.$$
Therefore, by Lemma \ref{Lem:cota heat L} and a similar argument as in the proof of Proposition \ref{Prop:Maximal}\textit{(i)}, for every $N>0$,
$$\abs{\K_\gamma(x,y)}\leq\frac{C}{\abs{x-y}^{n-\gamma}}\left(1+\frac{\abs{x-y}}{\rho(x)}+\frac{\abs{x-y}}{\rho(y)}\right)^{-N}.$$
In particular, $\L^{-\gamma/2}$ is bounded from $L^p(\Real^n)$ into $L^q(\Real^n)$, for $\tfrac{1}{q}=\frac{1}{p}-\tfrac{\gamma}{n}$ with $1<p<q<\infty$ and $0<\gamma<n$. Using similar arguments to those in the proof of Proposition \ref{Prop:Maximal}\textit{(ii)} it can be checked that
$$\abs{\K_\gamma(x,y)-\K_\gamma(x,z)}+\abs{\K_\gamma(y,x)-\K_\gamma(z,x)}\leq C\frac{|y-z|^\delta}{|x-y|^{n-\gamma+\delta}},$$
when $|x-y|>2|y-z|$,  for any $0<\delta<{2-\frac{n}{q}}$. Thus $\L^{-\gamma}$ is a $\gamma$-Schr\"odinger-Calder\'on-Zygmund operator according to Definition \ref{Defn:Operators}.

The second item of Theorem \ref{Thm:Pola} is a consequence of the following proposition and our two main theorems.

\begin{prop}
Let $B=B(x,s)$ with $0<s\leq\tfrac{1}{2}\rho(x)$. Then
\begin{enumerate}[(i)]
    \item $\displaystyle\log\left(\frac{\rho(x)}{s}\right)\frac{1}{|B|^{1+\frac{\gamma}{n}}}\int_B|\L^{-\gamma/2}1(y)-(\L^{-\gamma/2}1)_B|~dy\leq C$ if $\gamma\le{2-\frac{n}{q}}$;
    \item $\displaystyle\left(\frac{\rho(x)}{s}\right)^\alpha\frac{1}{|B|^{1+\frac{\gamma}{n}}}\int_{B}|\L^{-\gamma/2}1(y)-(\L^{-\gamma/2}1)_B|~dy\leq
    C$ if $\alpha+\gamma<\min\{1,2-\frac{n}{q}\}$.
\end{enumerate}
\end{prop}

\begin{proof}
Fix $y,z\in B$, so that $\rho(x)\sim\rho(y)\sim\rho(z)$. We can write
\begin{equation}\label{difference}
\L^{-\gamma/2}1(y)-\L^{-\gamma/2}1(z)=\int_0^\infty\int_{\Real^n}\left(\W_t(y,w)-\W_t(z,w)\right)\,dw~t^{\gamma/2}\,\frac{dt}{t}.
\end{equation}
We split the integral in $t$ of the difference \eqref{difference} into two parts. From \eqref{eq6} we have
$$\abs{\int_0^{\rho(x)^2}\int_{\Real^n}\left(\W_t(y,w)-\W_t(z,w)\right)\,dw~t^{\gamma/2}\,\frac{dt}{t}}
\leq C\left(\frac{s}{\rho(x)}\right)^\delta\int_0^{\rho(x)^2}t^{\gamma/2}~\frac{dt}{t}= C\left(\frac{s}{\rho(x)}\right)^\delta\rho(x)^\gamma.$$
On the other hand we can use \eqref{eq5ch} to get
$$\abs{\int_{\rho(x)^2}^\infty\int_{\Real^n}\left(\W_t(y,w)-\W_t(z,w)\right)\,dw~t^{\gamma/2}\,\frac{dt}{t}}\leq C\int_{\rho(x)^2}^\infty\left(\frac{s}{\sqrt{t}}\right)^{\delta}t^{\gamma/2}~\frac{dt}{t}\leq C\left(\frac{s}{\rho(x)}\right)^\delta\rho(x)^\gamma,$$
since $\gamma<\delta$. An application of these last two estimates to \eqref{difference} finally gives
\begin{multline*}
    \frac{1}{|B|^{1+\frac{\gamma}{n}}}\int_{B}|\L^{-\gamma/2}1(y)-(\L^{-\gamma/2}1)_B|dy \\
    \leq \frac{C}{s^{2n+\gamma}}\int_B\int_B|\L^{-\gamma/2}1(y)-\L^{-\gamma/2}1(z)|\,dy\,dz \leq C\left(\frac{s}{\rho(x)}\right)^{\delta-\gamma}.
\end{multline*}
Thus \textit{(i)} is valid if {$\gamma<2-\frac{n}{q}$ and $\delta<2-\frac{n}{q}$ is chosen such that $\gamma\leq\delta$. Also} \textit{(ii)} holds when $\alpha+\gamma<\min\{1,2-\frac{n}{q}\}$.
\end{proof}

\medskip

\noindent\textbf{Acknowledgements.} The second author wishes to thank the Departamento de Matem\'aticas y Computaci\'on of Universidad de La Rioja, Spain, for their kind hospitality.

\end{document}